\documentclass{jcmlatex}
\def\JCMvol{}
\def\JCMno{}
\def\JCMyear{}
\def\JCMreceived{}
\def\JCMrevised{}
\def\JCMaccepted{}

\usepackage{amsfonts,amssymb,amsmath,bm}
\usepackage{graphicx,epsfig}
\usepackage{epstopdf}
\usepackage{float}
\usepackage{algorithmic}
\usepackage{algorithm}
\usepackage{enumerate}
\usepackage{multirow}
\usepackage{xcolor}

\DeclareMathOperator\dist{dist}

\def\N{\mathbb{N}}
\def\R{\mathbb{R}}
\def\s{\mathbf{s}}
\def\x{\mathbf{x}}
\def\y{\mathbf{y}}

\begin{document}

\markboth{W.~LIU, Z.Q.~XIE, AND W.F.~YI}{Nonmonotone Local Minimax Methods for Finding Multiple Saddle Points}

\title{NONMONOTONE LOCAL MINIMAX METHODS FOR FINDING MULTIPLE SADDLE POINTS}

\author{Wei Liu\thanks{
South China Research Center for Applied Mathematics and Interdisciplinary Studies, South China Normal University, Guangzhou 510631, China; \\
Key Laboratory of Computing and Stochastic Mathematics (Ministry of Education), Hunan Normal University, Changsha, Hunan 410081, China. \\ 
Email: wliu@m.scnu.edu.cn}
\and
Ziqing Xie\footnote{Corresponding author.}
\thanks{Key Laboratory of Computing and Stochastic Mathematics (Ministry of Education), Hunan Normal University, Changsha, Hunan 410081, China. \\
Email: ziqingxie@hunnu.edu.cn}
\and
Wenfan Yi\thanks{Hunan Provincial Key Laboratory of Intelligent Information Processing and Applied Mathematics, School of Mathematics, Hunan University, Changsha, Hunan 410082, China. \\
Email: wfyi@hnu.edu.cn}
}

\maketitle

\begin{abstract}
In this paper, by designing a normalized nonmonotone search strategy with the Barzilai--Borwein-type step-size, a novel local minimax method (LMM), which is a globally convergent iterative method, is proposed and analyzed to find multiple (unstable) saddle points of nonconvex functionals in Hilbert spaces. Compared to traditional LMMs with monotone search strategies, this approach, which does not require strict decrease of the objective functional value at each iterative step, is observed to converge faster with less computations. Firstly, based on a normalized iterative scheme coupled with a local peak selection that pulls the iterative point back onto the solution submanifold, by generalizing the Zhang--Hager (ZH) search strategy in the optimization theory to the LMM framework, a kind of normalized ZH-type nonmonotone step-size search strategy is introduced, and then a novel nonmonotone LMM is constructed. Its feasibility and global convergence results are rigorously carried out under the relaxation of the monotonicity for the functional at the iterative sequences. Secondly, in order to speed up the convergence of the nonmonotone LMM, a globally convergent Barzilai--Borwein-type LMM (GBBLMM) is presented by explicitly constructing the Barzilai--Borwein-type step-size as a trial step-size of the normalized ZH-type nonmonotone step-size search strategy in each iteration. Finally, the GBBLMM algorithm is implemented to find multiple unstable solutions of two classes of semilinear elliptic boundary value problems with variational structures: one is the semilinear elliptic equations with the homogeneous Dirichlet boundary condition and another is the linear elliptic equations with semilinear Neumann boundary conditions. Extensive numerical results indicate that our approach is very effective and speeds up the LMMs significantly.
\end{abstract}

\begin{classification}
65K10, 58E05, 49M37, 35J20.
\end{classification}

\begin{keywords}
multiple saddle points,
local minimax method,
Barzilai--Borwein gradient method,
normalized nonmonotone search strategy,
global convergence
\end{keywords}

\section{Introduction}

Let $X$ be a Hilbert space. The critical points of a continuously Fr\'{e}chet-differentiable functional $E:X\to\R$ are defined as solutions to the associated Euler--Lagrange equation
\[ E'(u)=0, \quad  u\in X, \]
where $E'$ is the Fr\'{e}chet-derivative of $E$. The first candidates of critical points are local minima and maxima on which traditional calculus of variations and optimization methods focus. Critical points that are not local extrema are unstable and called saddle points. When the second-order Fr\'{e}chet-derivative $E''$ exists at some critical point $u_*$, the instability of $u_*$ can be depicted by its Morse index (MI) \cite{Chang1993}. In fact, the MI of such a critical point $u_*$, denoted by $\mathrm{MI}(u_*)$, is defined as the maximal dimension of subspaces of $X$ on which the linear operator $E''(u_*)$ is negative-definite. In addition, $u_*$ is said to be nondegenerate if $E''(u_*)$ is invertible. For a nondegenerate critical point, if its $\mathrm{MI}=0$, it is a strict local minimizer and then a stable critical point, while if its $\mathrm{MI}>0$, it is a saddle point and then an unstable critical point. Generally speaking, the higher the MI is, the more unstable the critical point is.

Saddle points, as unstable equilibria or transient excited states, are widely found in numerous nonlinear problems in physics, chemistry, biology and materials science \cite{Chang1993,CZN2000IJBC,H1973AA,LWZ2013JSC,Rabinowitz1986,XYZ2012SISC,ZRSD2016npj}. They play an important role in many interesting applications, such as studying rare transitions between different stable/metastable states \cite{CLEZS2010PRL,ERV2002PRB} and predicting morphologies of critical nucleus in the solid-state phase transformation \cite{ZCD2007PRL,ZRSD2016npj}, etc.. Due to various difficulties in direct experimental observation, more and more attentions have been paid to develop effective and reliable numerical methods for catching saddle points. Compared with the computation of stable critical points, it is much more challenging to design a stable, efficient and globally convergent numerical method for finding saddle points due to the instability and multiplicity. In recent years, motivated by some early algorithms for searching for saddle points in computational physics/chemistry/biology, the dimer method \cite{HJ1999JCP,ZD2012SINUM}, the gentlest ascent dynamics \cite{EZ2011NL}, the climbing string method \cite{ERV2002PRB,RV2013JCP} and etc., have been proposed and successfully implemented to find saddle points. It is noted that methods mentioned above mainly consider saddle points with $\mathrm{MI}=1$.

With the development of science and technology, the stable numerical computation of multiple unstable critical points with high MI has attracted more and more attentions both in theories and applications. Studies of relevant numerical methods have been carried out in the literature. Inspired by the minimax theorems in the critical point theory (see, e.g., \cite{Rabinowitz1986}) and the work of Choi and McKenna \cite{CM1993NA}, Ding, Costa and Chen \cite{DCC1999NA} and Chen, Zhou and Ni \cite{CZN2000IJBC}, a local minimax method (LMM) was developed by Li and Zhou in \cite{LZ2001SISC,LZ2002SISC} with its global convergence established in \cite{LZ2002SISC,Z2017CAMC}. As shown in \cite{Zhou2005MC}, for a nondegenerate critical point found by the LMM, its MI is determined {\em a priori} by the dimension of the given support space $L$ (see the detail in Sect.~\ref{subsec:lmt}) as $\mathrm{MI}=\dim(L)+1$. Therefore, the LMM is capable of selectively finding the saddle points with any given $\mathrm{MI}=n\geq1$ by appropriately constructing the support space $L$ with $\dim(L)=n-1$. Then, in \cite{XYZ2012SISC}, Xie, Yuan and Zhou modified the LMM with a significant relaxation for the domain of the local peak selection, which is a vital definition for the LMM (see below), and provided the global convergence analysis for this modified LMM by overcoming the lack of homeomorphism of the local peak selection. More modifications and developments of the LMM for multiple solutions of various problems, such as elliptic partial differential equations (PDEs) with nonlinear boundary conditions, quasi-linear elliptic PDEs in Banach spaces, upper semi-differentiable locally Lipschitz continuous functional and so on, have been also studied. We refer to \cite{LWZ2013JSC,Y2013MC,YZ2005SISC,Z2017CAMC} and references therein for this topic. In addition to the LMM, there are other optimization-based approaches for finding saddle points in the literature. In \cite{GLZ2015SINUM}, Gao, Leng and Zhou proposed an iterative minimization formulation (IMF) for searching index-1 saddle points and generalized it to the high-index case. They further proposed an efficient algorithm of the IMF with some adaptive strategies \cite{GLZ2015SINUM}. Actually, at each cycle of the IMF, a local minimization subproblem near the current position has to be solved, and its objective function is locally constructed by transforming the given energy function according to the information of the current position and the minimal mode of the Hessian. It was proved in \cite{GLZ2015SINUM} that, under certain conditions, the IMF for finding a saddle point has a quadratic local convergence rate with respect to the number of the iterative minimization subproblems. In \cite{GZ2018JCP}, Gu and Zhou designed a convex splitting scheme for solving the minimization subproblem at each cycle of the IMF, which allows for large step-sizes. On the other hand, Xie and Chen et al. proposed a search extension method and its modified versions for finding multiple solutions of nonlinear elliptic equations based on generalized Fourier series expansion and homotopy method; see \cite{CX2004CMA,CX2008SCM,LXY2021SSM,XCX2005IMA}. In \cite{YZZ2019SISC}, Yin, Zhang and Zhang proposed a high-index optimization-based shrinking dimer method to find unconstrained high-index saddle points. Recently, a constrained gentlest ascent dynamics was introduced in \cite{LXY2022JCP} to find constrained saddle points with any specified MI and applied to study excited states of Bose--Einstein condensates.

Since our work in this paper is inspired by the traditional LMMs, we want to recall the history of them a little more. The idea of LMMs characterizes a saddle point with specified MI as a solution to a two-level nested local optimization problem consisting of inner local maximization and outer local minimization. Owing to this local minimax characterization, the corresponding numerical algorithm designed for finding multiple saddle points with $\mathrm{MI}\geq 1$ becomes possible. In practical computations, the inner local maximization is equivalent to an unconstrained optimization problem in an Euclidean space with a fixed dimension, and thus many standard optimization algorithms can be employed to solve it efficiently. While the outer local minimization is more challenging because it solves an optimization problem in an infinite-dimensional submanifold (as described later) and is the main concern to the LMMs' family. In \cite{LZ2001SISC}, an exact step-size rule combined with a normalized gradient descent method was introduced for the outer local minimization, and then the LMM was successfully applied to solve a class of semilinear elliptic boundary value problems (BVPs) for multiple unstable solutions. However, the exact step-size rule is not only expensive in the numerical implementation, but also inconvenient for establishing convergence properties \cite{LZ2002SISC}. To compensate for these shortcomings, in the subsequent work \cite{LZ2002SISC}, Li and Zhou introduced a new step-size rule that borrowed the idea of the Armijo or backtracking line search rule \cite{Armijo1966} prefered in the optimization theory. Thanks to this step-size rule, the global convergence result for the LMM was obtained in \cite{LZ2002SISC,Z2017CAMC}. Then, a normalized Goldstein-type LMM was put forward in our work in \cite{LXY2021CMS} by recommending a normalized Goldstein step-size search rule that guarantees the sufficient decrease of the energy functional and prevents the step-size from being too small simultaneously. The feasibility and global convergence analysis of this approach was also provided. Recently, as a subsequent work \cite{LXY-NWPLMM}, both normalized Wolfe--Powell-type and strong Wolfe--Powell-type step-size search rules for the LMM were introduced and global convergence of the corresponding algorithms for general decrease directions were verified. The LMM with general descent direction provides the possibility for accelerating the convergence by appropriately choosing the descent direction. Above all, up to now the algorithm design and convergence analysis of monotonically decreasing LMMs with several typically normalized inexact step-size search rules have been systematically investigated. However, it should be pointed out that, so far all LMMs have to guarantee the sufficient decrease of the objective functional, and then the normalized monotone step-size search rules are crucial. These strict requirement of the monotone decrease may lead to expensive computations and make the convergence slower.

In this paper, we focus on proposing a novel LMM with fast convergence for finding multiple saddle points by developing suitable normalized nonmonotone step-size search strategies to replace normalized monotone step-size search rules utilized in traditional LMMs. Let us recall that, in optimization theory, the Barzilai--Borwein (BB) method developed by Barzilai and Borwein in 1988 \cite{BB1988IMANUM} is an efficient nonmonotone method and does not require the decrease of the objective function value at each iteration, distinguished from monotone methods. Actually, it can be regarded as a gradient method with modified step-sizes, which is enlightened by the idea of the quasi-Newton method for avoiding matrix computations. Compared with the classical steepest descent method put forward by Cauchy in \cite{Cauchy1847}, the BB method requires less computational efforts and often speeds up the convergence significantly \cite{BB1988IMANUM,Fletcher2005}. Recently, some adaptive BB methods in optimization theory were developed \cite{DHL2019COA,HDL2021SIOPT} and generally performed better than the conventional BB methods. To ensure the global convergence, the BB method is usually integrated with nonmonotone line search strategies to minimize general smooth functions in the optimization theory \cite{R1997SIOPT,WY2013MP}. It should be pointed out that, if a nonmonotone search strategy is used, occasional growth of the function value during the iteration is permitted. Two popular nonmonotone search strategies, i.e., the Grippo--Lampariello--Lucidi (GLL) search strategy and Zhang--Hager (ZH) search strategy were introduced in \cite{GLL1986SINUM} and \cite{ZH2004SIOPT}, respectively. By combining the GLL nonmonotone search strategy with the BB method, Raydan introduced an efficient globalized BB method for large-scale unconstrained optimization problems in \cite{R1997SIOPT}. The relevant algorithm is competitive with some standard conjugate gradient methods \cite{R1997SIOPT}. The ZH nonmonotone search strategy \cite{ZH2004SIOPT}, since it was proposed, has also been widely used in the literature to implement the globalized BB method; see, e.g., \cite{WY2013MP}. Recently, Zhang et al. \cite{YZZ2019SISC,ZDZ2016SISC} united the BB method with the shrinking dimer method \cite{ZD2012SINUM} to efficiently find saddle points with specified MI without theoretical proof for the global convergence yet. Besides, as far as we know, there is no more work applying the BB method for finding saddle points.

Inspired by traditional LMMs, the BB method and its globalization strategies founded on nonmonotone step-size searches in the optimization theory, this paper is devoted to establish globally convergent nonmonotone LMMs for finding multiple saddle points of nonconvex functionals in Hilbert spaces. Actually, a normalized ZH-type nonmonotone step-size search rule, which can be viewed as a relaxation to the normalized monotone Armijo-type step-size search rule, will be introduced to the LMM. Its feasibility will be verified. However, the available global convergence analysis for the traditional monotone LMMs in which the monotone behavior plays an important role may not be directly applicable to the newly proposed nonmonotone ones. Fortunately, the global convergence of the normalized ZH-type nonmonotone LMM will be carried out rigorously by the usage of the monotonicity of a convex combination of the functional values at past iterative steps and some nonlinear functional analysis tools combined with the compactness and the proof by contradiction. Then, a globally convergent Barzilai--Borwein-type LMM (GBBLMM) will be presented by explicitly constructing the BB-type step-size as a trial step-size in each iteration to speed up the LMM. It is worthwhile to point out that, like its LMM family members, the GBBLMM is also able to find saddle points in a stable way and avoid repeatedly being convergent to previously found saddle points. Finally, the GBBLMM will be numerically applied and compared with traditional LMMs in finding multiple unstable solutions of various problems. For example, the Lane-Emden equation, H\'{e}non equation in the astrophysics and the linear elliptic equations with semilinear Neumann boundary conditions which widely appear in many scientific fields, such as corrosion/oxidation modelings, metal-insulator or metal-oxide semiconductor systems.

The rest of this paper is organized as follows. In Sect.~\ref{sec:pre}, the preliminaries for the LMM and the BB method in the optimization theory are provided. In Sect.~\ref{sec:nmlmm}, the normalized ZH-type nonmonotone step-size rule is introduced to the LMM and its corresponding feasibility and global convergence are analyzed. Then, the GBBLMM is presented in Sect.~\ref{sec:gbblmm} by constructing an explicit BB-type trial step-size of the nonmonotone step-size search rule. Further, in Sect.~\ref{sec:numer}, extensive numerical results are stated and compared with traditional LMMs. Finally, some concluding remarks are provided in Sect.~\ref{sec:con}.

\section{Preliminaries}
\label{sec:pre}
In this section, for the convenience of discussion later, we revisit the basic ideas and algorithm framework of the LMM as well as the BB method in the optimization theory.

\subsection{Local minimax principle}\label{subsec:lmt}

Throughout this paper, we assume that $E$ has a local minimizer at $0\in X$ and focus on finding nontrivial saddle points of $E$. We begin with some notations and basic lemmas.

Let $(\cdot,\cdot)$ and $\|\cdot\|$ be the inner product and norm in $X$ respectively, $\langle\cdot,\cdot\rangle$ the duality pairing between $X$ and its dual space $X^*$, $S=\{v\in X:\|v\|=1\}$ the unit sphere in $X$, and $Y^\bot$ the orthogonal complement to $Y\subset X$. Suppose that $L\subset X$ is a finite-dimensional closed subspace and serves as the so-called support or base space \cite{LZ2001SISC,LZ2002SISC,XYZ2012SISC}. Define the half subspace $[L, v]=\{tv+w^L:t\geq0,w^L\in L\}$ for any $v\in S$.
\begin{definition}[\cite{XYZ2012SISC}]\label{def:pv}
	Denote $2^X$ as the set of all subsets of $X$. A {\em peak mapping} of $E$ w.r.t. $L$ is a set-valued mapping $P:X\to 2^X$ such that
	\[ P(v)=\left\{u\in X:\mbox{$u$ is a local maximizer of $E$ on $[L, v]$}\right\},\quad\forall\, v\in S. \]
	A {\em peak selection} of $E$ w.r.t. $L$ is a mapping $p:S\to X$ such that $p(v)\in P(v)$, $\forall\, v\in S$. For a given $v\in S$, we say that $E$ has a {\em local peak selection} w.r.t. $L$ at $v$, if there exists a neighborhood $N_v$ of $v$ and a mapping $p:N_v\cap S\to X$ such that $p(u)\in P(u)$, $\forall\, u\in N_v\cap S$.
\end{definition}

Set $p(v)$ to be a local peak selection of $E$ w.r.t. $L$ at $v=v^\bot+v^L\in S\backslash L$ with $v^\bot\in L^\bot\backslash\{0\}$ and $v^L\in L$. Since $p(v)\in[L,v]$, it can be expressed as $p(v)=t_vv+w_v^L$ with $t_v\geq0$ and $w_v^L\in L$. In the subsequent analysis, we assume $t_v>0$, i.e., $p(v)\notin L$, to avoid the degeneracy. Hereafter, we use the notation $A\backslash B=\{a\in A: a\notin B\}$ to denote the relative complement of the set $B$ in the set $A$. The definition of the local peak selection obviously leads to the following property.
\begin{lemma}[\cite{XYZ2012SISC}]\label{lem:orth}
	Suppose $E\in C^1(X,\R)$ and let $p(v)$ be a local peak selection of $E$ w.r.t. $L$ at $v\in S\backslash L$ satisfying $p(v)\notin L$, then $\langle E'(p(v)),p(v)\rangle=0$ and $\langle E'(p(v)),w\rangle=0$, $\forall\, w\in[L, v]$.
\end{lemma}

Lemma~\ref{lem:orth} implies that each local peak selection belongs to the Nehari manifold
\[ \mathcal{N}_E=\{u\in X\backslash\{0\}:\langle E'(u),u\rangle=0\}, \]
which contains all nontrivial critical points of $E$. Denote $g=\nabla E(p(v))\in X$ as the gradient of $E$ at $p(v)$, i.e., the canonical dual or Riesz representer of $E'(p(v))\in X^*$, which is defined by
\[ (g,\phi)=\langle E'(p(v)),\phi\rangle,\quad \forall\, \phi\in X. \]
Obviously, Lemma~\ref{lem:orth} yields $g\in[L, v]^\bot$ if $p(v)\notin L$. In addition, for each $\alpha>0$, there is a unique orthogonal decomposition for $v(\alpha)={(v-\alpha g)}/{\|v-\alpha g\|}$ as
\begin{equation}\label{eq:va-dcom}
	v(\alpha)=\frac{v-\alpha g}{\|v-\alpha g\|} =\frac{v-\alpha g}{\sqrt{1+\alpha^2\|g\|^2}}
	=v^L(\alpha)+v^\bot(\alpha),
\end{equation}
with
\begin{equation}\label{eq:vaL}
	v^L(\alpha)=\frac{v^L}{\sqrt{1+\alpha^2\|g\|^2}}\in L,\quad
	v^\bot(\alpha)=\frac{v^\bot-\alpha g}{\sqrt{1+\alpha^2\|g\|^2}}\in L^\bot.
\end{equation}
Two lemmas below follow from direct calculations and their proofs are referred to those of Lemmas~3.2-3.4 in \cite{LXY2021CMS}.
\begin{lemma}[\cite{LXY2021CMS}]\label{lem:vs}
	For $v(\alpha)$ expressed in \eqref{eq:va-dcom} with $v=v^\bot+v^L\in S\backslash L$, $v^\bot\in L^\bot\backslash\{0\}$, $v^L\in L$ and $g=\nabla E(p(v))\in[L,v]^\perp$, there hold that $v(\alpha)\in S\backslash L$, $\|v^L(\alpha)\|\leq\|v^L\|<1$ and $\|v^\bot(\alpha)\|\geq\|v^\bot\|>0$, $\forall\,\alpha>0$.
	Further, if $g\neq0$, then
	\begin{equation}\label{eq:va-v}
		\frac{\alpha\|g\|}{\sqrt{1+\alpha^2\|g\|^2}}<\|v(\alpha)-v\|<\alpha\|g\|, \quad\forall\,\alpha>0,
	\end{equation}
	and
	\[ \lim_{\alpha\to0^+}\frac{\|v(\alpha)-v\|}{\alpha\|g\|}=1. \]
\end{lemma}
\begin{lemma}[\cite{LXY2021CMS}]\label{lem:pv-tv}
	Let $p$ be a local peak selection of $E$ w.r.t. $L$ at $\bar{v}\in S\backslash L$. For $\forall\, v\in S$ near $\bar{v}$, denote $p(v)=t_vv+w_v^L$ with $t_v\geq0$ and $w_v^L\in L$. If $p$ is continuous at $\bar{v}$, then the mappings $v\mapsto t_v$ and $v\mapsto w_v^L$ are continuous at $\bar{v}$.
\end{lemma}

In view of Lemmas~\ref{lem:orth}-\ref{lem:pv-tv} and in accordance with the lines in the proof of Lemmas~3.6-3.7 in \cite{LXY2021CMS}, one can obtain the following result, which is actually an improved version of Lemma~2.1 in \cite{LZ2001SISC} and Lemma~2.13 in \cite{Y2013MC} since the domain of the local peak selection is changed from $S\cap L^\bot$ to $S$.
\begin{lemma}[\cite{LXY2021CMS}]\label{lem:armijo}
	Suppose $E\in C^1(X,\R)$ and let $p(v)=t_vv+w_v^L$ be a local peak selection of $E$ w.r.t. $L$ at $v\in S\backslash L$, where $t_v\geq0$ and $w_v^L\in L$. If there hold (i) $p$ is continuous at $v$; (ii) $t_v>0$; and (iii) $g=\nabla E(p(v))\neq0$, then for any $\sigma\in(0,1)$, there exists $\alpha^A>0$ such that
	\[ E(p(v(\alpha)))< E(p(v))-\sigma\alpha t_v\|g\|^2,\quad \forall \,\alpha\in(0,\alpha^A). \]
\end{lemma}

Thanks to Lemma~\ref{lem:armijo}, the following result can be obtained by imitating the lines of the proof of Theorem~2.1 in \cite{LZ2001SISC} with its proof omitted here for simplicity.
\begin{theorem}\label{thm:lmt0}
	If $E\in C^1(X,\R)$ has a local peak selection w.r.t. $L$ at $v_*\in S\backslash L$, denoted by $p(v_*)=t_{v_*}v_*+w_{v_*}^L$, satisfying (i) $p$ is continuous at $v_*$; (ii) $t_{v_*}>0$; and (iii) $v_*$ is a local minimizer of $E(p(v))$ on $S\backslash L$, then $p(v_*)\notin L$ is a critical point of $E$.
\end{theorem}

The following concept of compactness is needed to establish an existence result.
\begin{definition}[\cite{Rabinowitz1986}]
	A functional $E\in C^1(X,\R)$ is said to satisfy Palais--Smale (PS) condition if every sequence $\{w_j\}\subset X$ such that $\{E(w_j)\}$ is bounded and $E'(w_j)\to0$ in $X^*$ has a convergent subsequence.
\end{definition}

For a given $v_0=v_0^\bot+v_0^L\in S\backslash L$ with $v_0^\bot\in L^\bot\backslash\{0\}$ and $v_0^L\in L$, define
\begin{equation}\label{eq:V0}
	\mathcal{V}_0:=\{v=v^\bot+\tau v_0^L\in S:v^\bot\in L^\bot,0\leq\tau\leq1\}\subset S\backslash L.
\end{equation}
As illustrated below in Lemma~\ref{lem:mono}, the sequence $\{v_k\}$ generated by the LMM algorithm with initial data $v_0$ is contained in $\mathcal{V}_0$. Actually, it will be seen that the domain of the peak selection $p$ can be limited to be the closed subset $\mathcal{V}_0$ instead of $S$.

Similar to Theorem~2.2 in \cite{LZ2001SISC}, by applying the Ekeland's variational principle and Lemma~\ref{lem:armijo}, we have the existence result with the proof given in Appendix~\ref{app:prf-lmt}, which is really an improvement for that in \cite{LZ2001SISC}. In fact, a continuous peak selection $p$ defined on $\mathcal{V}_0$ instead of $S\cap L^\bot$, in general, is no longer a homeomorphism.
\begin{theorem}\label{thm:lmt}
	Let $E\in C^1(X,\R)$ satisfy the (PS) condition. If $E$ has a peak selection w.r.t. $L$, denoted by $p(v)=t_vv+w_v^L$ with $v\in \mathcal{V}_0$ and $w_v^L\in L$, satisfying (i) $p$ is continuous on $\mathcal{V}_0$; (ii) $t_v\geq\delta$ for some $\delta>0$ and $\forall\,v\in \mathcal{V}_0$; and (iii) $\inf_{v\in \mathcal{V}_0}E(p(v))>-\infty$, then there exists $v_*\in \mathcal{V}_0$ such that $p(v_*)\notin L$ is a critical point and
	\[ E(p(v_*))=\inf_{v\in \mathcal{V}_0}E(p(v)). \]
\end{theorem}

Under assumptions of Theorem~\ref{thm:lmt}, there is a saddle point, as an unstable critical point of $E$, characterized as a local solution to the constrained local minimization problem
\begin{equation}\label{eq:LMM2}
	\min_{v\in \mathcal{V}_0}E(p(v))\quad \mbox{or}\quad\min_{w\in\mathcal{M}}E(w),
\end{equation}
with $\mathcal{M}=\{p(v):v\in \mathcal{V}_0\}$ serving as the `solution submanifold'. Thus, some descent algorithms can work for numerically finding saddle points of the functional $E$ determined by \eqref{eq:LMM2} in a stable way. Consequently, Theorems~\ref{thm:lmt0} and \ref{thm:lmt} provide mathematical justifications for the LMM.

Last but not least, the LMMs are capable of selectively finding the saddle points with given MIs. Following the lines of Theorem~2.4 in \cite{Zhou2005MC}, we can prove that, under some conditions, the MI of a nondegenerate critical point $u_*=p(v_*)\notin L$ characterized by Theorem~\ref{thm:lmt0} or \ref{thm:lmt} is given as
\[ \mathrm{MI}(u_*)=\dim(L)+1. \]
Thus, $L=\{0\}$, as the simplest case, usually leads to a saddle point with $\mathrm{MI}=1$. In general, to find a saddle point with a given $\mathrm{MI}=n>1$ by the LMMs, the $(n-1)$-dimensional support space $L$ needs to be prescribed. In our numerical experiments, $L$ is simply spanned by some of previously found critical points.

\subsection{Local minimax algorithm}\label{subsec:lma}

To numerically find multiple saddle points in a stable way, traditional LMMs solve the constrained local minimization problem \eqref{eq:LMM2} via the following iterative scheme
\[ v_{k+1}=v_k(\alpha_k):=\frac{v_k-\alpha_kg_k}{\|v_k-\alpha_kg_k\|},\quad w_{k+1}=p(v_{k+1}),\quad k=0,1,\ldots, \]
where $v_k\in S\backslash L$, $\alpha_k>0$ is called a step-size and $g_k=\nabla E(w_k)$ denotes the gradient of $E$ at $w_k=p(v_k)$. A generalized framework of the LMM algorithm is outlined in Algorithm~\ref{alg:lmm} and we refer to \cite{LZ2001SISC,LZ2002SISC,XYZ2012SISC} for more details.

\begin{algorithm}[!ht]
	\caption{Generalized framework of the algorithm for the LMM \cite{LZ2001SISC,LZ2002SISC,XYZ2012SISC}.
	}\normalsize
	\label{alg:lmm}
	\begin{enumerate}[\bf Step 1.]
		\item Let the support space $L$ be spanned by some previously found critical points of $E$, say $u_1,u_2,\ldots,u_{n-1}\in X$, where $u_{n-1}$ is assumed to have the highest energy functional value. The initial ascent direction $v_0\in S\backslash L$ at $u_{n-1}$ is given. Set $t_{-1}=1$, $w_{-1}^L=u_{n-1}$ and $k:=0$. Repeat {\bf Steps~2-4} until the stopping criterion is satisfied (e.g., $\|\nabla E(w_k)\|\leq \varepsilon_{\mathrm{tol}}$ for a given tolerance $0<\varepsilon_{\mathrm{tol}}\ll 1$), then output $u_n=w_k$.
		\item Using the initial guess $w=t_{k-1}v_k+w_{k-1}^L\in[L,v_k]$, solve for
		\[ w_k=p(v_k)\equiv t_kv_k+w_k^L=\arg\max_{w\in[L, v_k]}E(w). \]
		\item Compute the gradient $g_k=\nabla E(w_k)\in X$ by solving the linear subproblem
		\begin{equation}\label{eq:gk}
			(g_k,\phi)=\langle E'(w_k),\phi\rangle,\quad \forall\,\phi\in X.
		\end{equation}
		\item Choose a suitable step-size $\alpha_k>0$ to update
		\[ v_{k+1}=v_k(\alpha_k)=\frac{v_k-\alpha_k g_k}{\|v_k-\alpha_k g_k\|}.\]
		Set $k:=k+1$ and go to {\bf Step~2}.
	\end{enumerate}
\end{algorithm}

A significant behavior of the iterative sequence is presented in the following lemma and can be proved using the expression $v_{k+1}=v_k(\alpha_k)$ and \eqref{eq:va-dcom}-\eqref{eq:vaL}, and the details can be found in Lemma~2.3 in \cite{XYZ2012SISC}.
\begin{lemma}[\cite{XYZ2012SISC}]\label{lem:mono}
	Let $\{v_k\}$ be a sequence generated by  Algorithm~\ref{alg:lmm} with $v_0\in S\backslash L$. Denote $v_k=v_k^\bot+v_k^L$ with $v_k^\bot\in L^\bot$ and $v_k^L\in L$, $k=0,1,\ldots$,  then $\|v_0^\bot\|\leq\|v_k^\bot\|\leq1$ and $v_k^L=\tau_kv_0^L$ hold for $0<\tau_{k+1}\leq\tau_k\leq1$, $k=0,1,\ldots$.
\end{lemma}

This implies that once an initial guess $v_0\in S\backslash L$ is used in Algorithm~\ref{alg:lmm}, we can focus our attention on the domain of a peak selection $p$ only on the closed subset $\mathcal{V}_0$ defined in \eqref{eq:V0}, which contains all possible $v_k$ that the algorithm may generate.

Next, we describe the following weaker version of the homeomorphism of $p$, which plays an essential role for the global convergence in Sect.~\ref{sec:nmlmm}. The proof is similar to that of Theorem~2.1 in \cite{XYZ2012SISC} and skipped here for brevity.
\begin{lemma}\label{lem:home}
	Suppose $E\in C^1(X,\R)$ and let $p$ be a peak selection of $E$ w.r.t. $L$ and $\{v_k\}$ be a sequence generated by Algorithm~\ref{alg:lmm} with $v_0\in S\backslash L$. Denote $w_k=p(v_k)=t_kv_k+w_k^L$ with $t_k\geq0$ and $w_k^L\in L$. Assume that (i) $p$ is continuous on $\mathcal{V}_0$ and (ii) $t_k\geq\delta$ for some $\delta>0$ and $\forall\,k=0,1,\ldots$ hold. If $\{w_k\}$ contains a subsequence $\{w_{k_i}\}$ converging to some $u_*\in X$, then the corresponding subsequence $\{v_{k_i}\}$ converges to some $v_*\in \mathcal{V}_0$ satisfying $u_*=p(v_*)$.
\end{lemma}

It is worthwhile to point out that step-size search rules used in traditional LMMs include the optimal/exact step-size search rule \cite{LZ2001SISC}, the normalized Armijo-type step-size search rule \cite{LZ2002SISC,YZ2005SISC,XYZ2012SISC}, the normalized Glodstein-type step-size search rule \cite{LXY2021CMS}  and the normalized Wolfe--Powell-type step-size search rules \cite{LXY-NWPLMM}. And up to now all step-size search rules in traditional LMMs are monotone in the sense that the sequence $\{E(w_k)\}$ is monotonically decreasing. Actually this feature is vital for the convergence analysis in traditional LMMs; see \cite{LZ2002SISC,Z2017CAMC,XYZ2012SISC,LXY2021CMS,LXY-NWPLMM}. Since the work of this paper is closely related to the normalized Armijo-type step-size search rule, let us describe it a little more. In fact, if the normalized Armijo-type step-size search rule is employed in Algorithm~\ref{alg:lmm}, the step-size $\alpha_k$ is chosen by a backtracking strategy as \cite{LZ2002SISC,YZ2005SISC,XYZ2012SISC}
\begin{equation}\label{eq:ak-armijo}
	\alpha_k = \max\left\{\lambda\rho^m:\; m\in\N,\, E(p(v_k(\lambda\rho^m))) \leq E(p(v_k)) - \sigma\lambda\rho^mt_k\|g_k\|^2\right\},
\end{equation}
for $k=0,1,\ldots$, where $g_k=\nabla E(p(v_k))$ and $\sigma,\rho\in(0,1)$, $\lambda>0$ are given parameters.

To end this subsection, let us explore a key property associated with the normalized Armijo-type step-size search rule, which will be quite useful in the convergence analysis in Sect.~\ref{sec:nmlmm}. According to Lemma~\ref{lem:armijo}, it is reasonable to define {\em the largest normalized Armijo-type step-size} at $v\in \mathcal{V}_0$ under the assumptions in Lemma~\ref{lem:armijo} as
\begin{equation}\label{eq:maxAstep}
	\bar{\alpha}^A(v) := \sup\left\{\alpha>0: E(p(v(\alpha)))<E(p(v))-\sigma\alpha t_v\|g\|^2 \right\}.
\end{equation}
The following lemma states that $\bar{\alpha}^A(v)$ is uniformly away from zero when $v$ is close to some point $\bar{v}\in \mathcal{V}_0$ such that $p(\bar{v})$ is not a critical point. The proof is similar to that of Lemma~2.5 in \cite{LZ2002SISC} and omitted here for brevity.
\begin{lemma}\label{lem:s0}
	Suppose $E\in C^1(X,\R)$ and let $p(\bar{v})=t_{\bar{v}}\bar{v}+w_{\bar{v}}^L$ be a local peak selection of $E$ w.r.t. $L$ at $\bar{v}\in \mathcal{V}_0$. If (i) $p$ is continuous at $\bar{v}$; (ii) $t_{\bar{v}}>0$; and (iii) $E'(p(\bar{v}))\neq0$ hold, then there exist a neighborhood $N_{\bar{v}}$ of $\bar{v}$ and a constant $\underline{\alpha}>0$ such that $\bar{\alpha}^A(v)\geq\underline{\alpha}$, $\forall\,v\in N_{\bar{v}}\cap \mathcal{V}_0$.
\end{lemma}

\subsection{BB method and nonmonotone globalization in the optimization theory}\label{subsec:bb}

In order to put forward our approach in Sect.~\ref{sec:nmlmm}-\ref{sec:gbblmm}, we now review ingenious ideas of the BB method and its nonmonotone globalization strategies used in the optimization theory. Consider an unconstrained minimization problem as
\begin{equation}\label{eq:minf}
	\min_{\x\in\R^d} f(\x),
\end{equation}
where $f$ is a continuously differentiable function defined on $\R^d$ (with the inner product $(\cdot,\cdot)_{\R^d}$ and norm $\|\cdot\|_{\R^d}$). The standard gradient method or the steepest descent method for solving \eqref{eq:minf} updates the approximate solution iteratively by
\begin{equation}\label{eq:sdi}
	\x_{k+1}=\x_k-\gamma_k\nabla f(\x_k),\quad k=0,1,\ldots,
\end{equation}
with a step-size $\gamma_k>0$ determined either by an exact or inexact line search. Although the steepest descent method is effective in practice, it may lead to a zigzag-like iterative path and its convergence is usually very slow \cite{SunYuan2006}. It is well known that the quasi-Newton method, which uses the iterative scheme as $\x_{k+1}=\x_k-\mathbf{B}_k^{-1}\nabla f(\x_k)$ with $\mathbf{B}_k$ an appropriate approximation to the Hessian matrix, has often a faster convergence than the steepest descent method since it inherits some merits of the Newton method \cite{SunYuan2006}. Unfortunately, the quasi-Newton method is very expensive for large-scale optimization problems since it involves matrix storage and computations in each iteration.

Rewriting \eqref{eq:sdi} as $\x_{k+1}=\x_k-\mathbf{D}_k\nabla f(\x_k)$ with $\mathbf{D}_k=\gamma_k\mathbf{I}$ and $\mathbf{I}$ the $d\times d$ identity matrix, in which $\mathbf{D}_k$ is regarded as an approximation to the inverse Hessian matrix, the BB method chooses $\gamma_k$ such that $\mathbf{D}_k$ approximately possesses certain quasi-Newton property \cite{BB1988IMANUM}, i.e.,
\begin{equation}\label{eq:sDy}
	\s_k\approx\mathbf{D}_k\y_k\quad\mbox{or}\quad \s_k\approx\gamma_k\y_k,\quad k=1,2,\ldots,
\end{equation}
where $\s_k=\x_{k}-\x_{k-1}$ and $\y_k=\nabla f(\x_{k})-\nabla f(\x_{k-1})$.
Solving \eqref{eq:sDy} in the least-squares sense, i.e., finding $\gamma_k$ to minimize $\|\s_k-\gamma\y_k\|_{\R^d}^2$, yields a BB step-size as
\begin{equation}\label{eq:RnBB1}
	\gamma_{k}^{BB1}=\frac{(\s_k,\y_k)_{\R^d}}{(\y_k,\y_k)_{\R^d}},\quad (\y_k,\y_k)_{\R^d}>0,\quad k=1,2,\ldots.
\end{equation}
According to the symmetry, one can alternatively minimize $\|\gamma^{-1}\s_k-\y_k\|_{\R^d}^2$ to obtain another BB step-size as
\begin{equation}\label{eq:RnBB2}
	\gamma_{k}^{BB2}=\frac{(\s_k,\s_k)_{\R^d}}{(\s_k,\y_k)_{\R^d}},\quad (\s_k,\y_k)_{\R^d}>0,\quad k=1,2,\ldots.
\end{equation}
In some sense, the BB method can be viewed as a very simple quasi-Newton method. So it may inherit the advantages of the quasi-Newton method with fast convergence without matrix operations. Actually, it is observed numerically in practice that the BB method often greatly speeds up the convergence of the gradient method \cite{BB1988IMANUM,Fletcher2005}.

However, due to essentially nonmonotone behaviors, there are potential difficulties in the convergence analysis for the BB method. In general, a globalization strategy founded on the nonmonotone line search is necessary for the BB method \cite{R1997SIOPT,SunYuan2006}. The basic idea is to use the step-size $\gamma_k=\beta_k\gamma_k^{BB}$ ($\gamma_k^{BB}=\gamma_k^{BB1}$ or $\gamma_k^{BB}=\gamma_k^{BB2}$) searched by a nonmonotone line search strategy, where the factor $\beta_k\in(0,1]$ plays the role of the step-size of the ``quasi-Newton" iteration: $\x_{k+1}=\x_k-\beta_k\mathbf{D}_k\nabla f(\x_k)$ with $\mathbf{D}_k=\gamma_k^{BB}\mathbf{I}$ for $k=1,2,\ldots$, and $\mathbf{D}_0=\gamma_0\mathbf{I}$ for a given $\gamma_0>0$. Recall that, as is one of popular nonmonotone line search strategies in the optimization theory, the ZH nonmonotone line search \cite{ZH2004SIOPT} is to find $\gamma_k^{ZH}=\lambda_k\rho^{m_k}$ with $m_k$ the smallest nonnegative integer satisfying
\begin{equation}\label{eq:nmlsZH}
	f(\x_k+\gamma_k^{ZH}\mathbf{d}_k) \leq C_k + \sigma\gamma_k^{ZH} (\nabla f(\x_k), \mathbf{d}_k)_{\R^d}, \quad k=0,1,\ldots,
\end{equation}
where $C_k$ is a weighted average of $\{f(\x_j), j=0,1,\ldots,k\}$, $\mathbf{d}_k\in\R^d$ denotes a descent direction at $\x_k$, $\lambda_k$ is a trial step-size and $\sigma,\rho\in(0,1)$ are given parameters. We remark here that the ideas of combining the BB method with nonmonotone globalization strategies in the optimization theory to speed up the convergence of the algorithm are that choosing $\mathbf{d}_k=-\nabla f(\x_k)$ and the trial step-size as the BB step-size, i.e.,  $\lambda_k=\gamma_k^{BB}$, for $k=1,2,\ldots$, explicitly.

\section{Nonmonotone LMM}
\label{sec:nmlmm}

In this section, in order to relax the restriction of the strict decrease of the objective functional value at each iterative step, we propose a kind of nonmonotone LMM by introducing the normalized ZH-type nonmonotone step-size search strategy for the Fr\'{e}chet-differentiable functionals on a Hilbert space. Further, some related properties are analyzed and the global convergence analysis are established. The same notations as those in Sect.~\ref{sec:pre} will be used unless specified.

In order to construct the normalized ZH-type nonmonotone step-size search rule for the LMM and establish its feasibility, the following lemma is needed.

\begin{lemma}\label{lem:ZHj}
	Suppose $E\in C^1(X,\R)$ and let $p(v)=t_vv+w_v^L$ with $t_v\geq0$ and $w_v^L\in L$ be a peak selection of $E$ w.r.t. $L$ at $v\in S$, and $k$ be some positive integer. Take $v_0\in S\backslash L$, $\sigma\in(0,1)$, $0\leq\eta_{\min}<\eta_{\max}\leq1$, $\eta_j\in[\eta_{\min},\eta_{\max}]$ and $\alpha_j>0$, $j=0,1,\ldots,k-1$. Set $Q_0=1$, $C_0=E(p(v_0))$, $t_j=t_{v_j}$, $g_j=\nabla E(p(v_j))$ and
	\begin{align*}
		v_{j+1}&=v_j(\alpha_j)=\frac{v_j-\alpha_jg_j}{\|v_j-\alpha_jg_j\|}, \quad
		Q_{j+1}=\eta_jQ_j+1,\\
		C_{j+1}&=(\eta_jQ_jC_j+E(p(v_{j+1})))/Q_{j+1},\quad j=0,1,\ldots,k-1.
	\end{align*}
	Assume that
	\begin{equation}\label{eq:ZHj}
		E(p(v_j(\alpha_j)))\leq C_j-\sigma\alpha_jt_j\|g_j\|^2,\quad j=0,1,\ldots,k-1.
	\end{equation}
	If (i) $p$ is continuous at $v_k$;
	(ii) $t_k=t_{v_k}>0$; and
	(iii) $g_k=\nabla E(p(v_k))\neq0$ hold, then there exists $\alpha_k^A>0$ such that
	\[ E(p(v_k(\alpha)))<C_k-\sigma\alpha t_k\|g_k\|^2,\quad \forall\,\alpha\in(0,\alpha_k^A). \]
\end{lemma}
\begin{proof}
	Denote $E_j=E(p(v_j))$, $j=0,1,\cdots,k$. From \eqref{eq:ZHj}, we have $E_{j+1}\leq C_j$ for $j=0,1,\ldots,k-1$. In particular, $E_k\leq C_{k-1}$. Hence
	\begin{equation}\label{eq:ZHj:eq1}
		C_k=(\eta_{k-1}Q_{k-1}C_{k-1}+E_k)/Q_k\geq (\eta_{k-1}Q_{k-1}E_k+E_k)/Q_k=E_k.
	\end{equation}
	Lemma~\ref{lem:armijo} states that there exists $\alpha_k^A>0$ such that
	\begin{equation}\label{eq:ZHj:eq2}
		E(p(v_k(\alpha)))<E_k-\sigma\alpha t_k\|g_k\|^2,\quad\forall\,\alpha\in(0,\alpha_k^A).
	\end{equation}
	The conclusion follows from the combination of \eqref{eq:ZHj:eq1} and \eqref{eq:ZHj:eq2}.
\end{proof}

Lemma \ref{lem:armijo} and Lemma~\ref{lem:ZHj} inspire us to define a normalized ZH-type nonmonotone step-size as follows.
\begin{definition}{\bf (Normalized ZH-type nonmonotone step-size)}\label{def-ZH}
	For $k=0,1,\ldots$, take $\sigma,\rho$ $\in(0,1)$, $0<\lambda_{\min}\leq\lambda_k\leq\lambda_{\max}<+\infty$, $0\leq\eta_{\min}\leq\eta_j\leq\eta_{\max}\leq1,\,j=0,1,\ldots,k-1$.
	If $\alpha=\lambda_k\rho^{m_k}$ and $m_k$ is the smallest positive integer satisfying
	\begin{equation}\label{eq:ZHcond}
		E(p(v_k(\alpha))) \leq C_k-\sigma\alpha t_k\|g_k\|^2,
	\end{equation}
	with $g_k=\nabla E(p(v_k))$, $Q_0=1$, $C_0=E(p(v_0))$ and
	\[ Q_j=\eta_{j-1} Q_{j-1}+1,\quad C_j=(\eta_{j-1} Q_{j-1}C_{j-1}+E(p(v_{j})))/Q_{j}, \quad j=1,2,\ldots,k, \]
	then we say that $\alpha$ is a normalized ZH-type nonmonotone step-size at $v_k$.
\end{definition}

Here, $\lambda_k\in[\lambda_{\min},\lambda_{\max}]$ is a trial step-size with the parameters $\lambda_{\min}$ and $\lambda_{\max}$ used to prevent the trial step-size from being too small or large. Reviving Algorithm~\ref{alg:lmm}, the algorithm of the LMM with the normalized ZH-type nonmonotone step-size search rule \eqref{eq:ZHcond} is described in Algorithm~\ref{alg:zhlmm}.

In the subsequent discussion in this subsection, we use the same notations and parameters as those in Algorithm~\ref{alg:zhlmm} unless specified. The feasibility of Algorithm~\ref{alg:zhlmm} is guaranteed by the following theorem and directly follows from Lemmas \ref{lem:armijo} and~\ref{lem:ZHj}.

\begin{algorithm}[!ht]
	\caption{Normalized ZH-type Nonmonotone Local Minimax Algorithm.}
	\label{alg:zhlmm}\normalsize
	Choose $\sigma,\rho\in(0,1)$, $0<\lambda_{\min}<\lambda_{\max}<+\infty$, $0\leq\eta_{\min}<\eta_{\max}\leq1$,  $Q_0=1$ and $C_0=E(p(v_0))$. {\bf Steps 1-3} are the same as those in Algorithm~\ref{alg:lmm}.
	\begin{enumerate}[\bf Step 1.]
		\setcounter{enumi}{3}
		\item Choose a trial step-size $\lambda_k\in[\lambda_{\min},\lambda_{\max}]$ and find
		\begin{equation}\label{eq:ak-zh}
			\alpha_k=\max_{m\in\N}\left\{\lambda_k\rho^m: E(p(v_k(\lambda_k\rho^m)))\leq C_k-\sigma\lambda_k\rho^mt_k\|g_k\|^2 \right\},
		\end{equation}
		where the initial guess $w=t_k v_k(\lambda_k\rho^m)+w_k^L$ is used to find the local maximizer $p(v_k(\lambda_k\rho^m))$ of $E$ on $[L, v_k(\lambda_k\rho^m)]$ for $m=0,1,\ldots$. \\[5pt]
		\hspace*{2em}Set $v_{k+1}=v_k(\alpha_k)$ and choose $\eta_k\in[\eta_{\min},\eta_{\max}]$ to calculate
		\begin{equation}\label{eq:QCupdate}
			Q_{k+1}=\eta_kQ_k+1,\quad C_{k+1}=\Big(\eta_kQ_kC_k+E(p(v_k(\alpha_k)))\Big)\big/Q_{k+1}.
		\end{equation}
		Update $k:=k+1$ and go to {\bf Step~2}.
	\end{enumerate}
\end{algorithm}

\begin{theorem}\label{thm:ZH}
	Assume that $E\in C^1(X,\R)$ has a peak selection $p$ of $E$ w.r.t. $L$. Let $\{v_j\}_{j=0}^k\subset \mathcal{V}_0$ be a sequence generated by Algorithm~\ref{alg:zhlmm} with $g_j\neq0$, $\forall\, j=0,1,\ldots,k$, for some $k\geq0$. If there hold (i) $p$ is continuous on $\mathcal{V}_0$ and (ii) $t_j>0$, $\forall\, j=0,1,\ldots,k$, then for each $j=0,1,\ldots,k$, there exists $\alpha_j^A>0$ such that
	\[ E(p(v_j(\alpha)))<C_j-\sigma\alpha t_j\|g_j\|^2,\quad \forall\,\alpha\in(0,\alpha_j^A). \]
\end{theorem}
\begin{proof}
	When $k=0$, the normalized ZH-type nonmonotone step-size search rule \eqref{eq:ZHcond} is exactly the normalized Armijo-type step-size search rule stated in \eqref{eq:ak-armijo}, and its feasibility is obvious from Lemma~\ref{lem:armijo}. For $k\geq1$, the assertion can be derived directly from Lemma~\ref{lem:ZHj} with an inductive argument on $j=0,1,\ldots,k$.
\end{proof}

In the following, we begin to consider the global convergence of Algorithm~\ref{alg:zhlmm}. Our analysis is based on some nonlinear functional analysis tools combined with the compactness and the proof by contradiction, and the key points are two aspects: (i) to prove that the sequence $\{v_k\}$ is a Cauchy sequence under the opposite assumptions made in order to derive the contradiction; and (ii) to establish a uniform lower bound for the step-size when the iterative point $p(v_k)$ approaches to a non-critical point. In order to handle these, we make full use of the monotonicity of $\{C_k\}$ described below and the weaker version of the homeomorphism of $p$ given in Lemma~\ref{lem:home}, and recognize the connection between the normalized ZH-type nonmonotone step-size \eqref{eq:ak-zh} and the largest normalized Armijo-type step-size \eqref{eq:maxAstep}.

For $j=0,1,\ldots$, denote $E_j=E(p(v_j))$,  then a direct calculation leads to
\begin{align}\label{eq:qkex}
	Q_{j+1} &= 1+\sum_{i=0}^j\left(\prod_{l=0}^i\eta_{j-l}\right) \leq j+2, \\ 
	\vspace{-1ex}
	\label{eq:ckex}
	C_{j+1} &= \frac{1}{Q_{j+1}}\left(E_{j+1}+\sum_{i=0}^j\left(\prod_{l=0}^i\eta_{j-l}\right)E_{j-i}\right).
\end{align}
Thus, $C_k$ is a convex combination of $\{E_j\}_{j=0}^k$ with large weights on recent $E_j$. We remark here that the choice of $\eta_j$ affects the degree of the nonmonotonicity of the normalized ZH-type nonmonotone step-size search rule \eqref{eq:ZHcond}. In fact, if $\eta_j=0$, $\forall j=0,1,\ldots,k-1$, then $Q_k=1$ and $C_k=E_k$. The normalized ZH-type nonmonotone step-size search rule \eqref{eq:ZHcond} is exactly the monotone normalized Armijo-type step-size search rule given in \eqref{eq:ak-armijo}; if $\eta_j=1$, $\forall j=0,1,\ldots,k-1$, then $Q_k=k+1$ and $C_k=A_k$ with $A_k := \frac{1}{k+1}\sum_{j=0}^kE_j$  the arithmetic mean of $\{E_j\}_{j=0}^k$. Actually, we have the following property, which can be verified by a similar argument in the proof of Lemma~1.1 in \cite{ZH2004SIOPT}. The proof is omitted here for simplicity.
\begin{lemma}\label{lem:ECA}
	The following inequalities hold for Algorithm~\ref{alg:zhlmm} under the same assumptions in Theorem~\ref{thm:ZH}, i.e.,
	\[ E_k\leq C_k\leq A_k\leq E_0,\quad k=0,1,\ldots. \]
\end{lemma}

Note that the following significant connection between the normalized ZH-type nonmonotone step-size \eqref{eq:ak-zh} and the largest normalized Armijo-type step-size \eqref{eq:maxAstep} is vital to establish the global convergence of Algorithm~\ref{alg:zhlmm}.
\begin{lemma}\label{lem:akgeq}
	Let $\{v_k\}\subset \mathcal{V}_0$ be a sequence generated by Algorithm~\ref{alg:zhlmm} and $\alpha_k$ be the normalized ZH-type nonmonotone step-size \eqref{eq:ak-zh} at $v_k$, then under the same assumptions in Theorem~\ref{thm:ZH}, we have
	\[ \alpha_k\geq\min\left\{\lambda_{\min},\rho\bar{\alpha}^A(v_k)\right\},\quad k=0,1,\ldots, \]
	where $\bar{\alpha}^A(v_k)$, defined in \eqref{eq:maxAstep}, is the largest normalized Armijo-type step-size at $v_k$.
\end{lemma}
\begin{proof}
	In fact, $\alpha_k=\lambda_k\rho^{m_k}$ for some $m_k\in\N$. If $m_k=0$, then $\alpha_k=\lambda_k\geq\lambda_{\min}$ and the conclusion holds. Otherwise, if $m_k>0$, the minimality of $m_k$ and Lemma~\ref{lem:ECA} lead to
	\[
	E(p(v_k(\rho^{-1}\alpha_k)))
	> C_k-\sigma\rho^{-1}\alpha_kt_k\|g_k\|^2
	\geq E_k-\sigma\rho^{-1}\alpha_kt_k\|g_k\|^2.
	\]
	By the definition of the largest normalized Armijo-type step-size $\bar{\alpha}^A(v_k)$ at $v_k$, one  can obtain $\alpha_k\geq\rho\bar{\alpha}^A(v_k)$ and the conclusions hold.
\end{proof}

Now, we are ready to consider the global convergence of Algorithm~\ref{alg:zhlmm}. Note that, by employing \eqref{eq:ak-zh} and \eqref{eq:QCupdate}, we have
\begin{equation}\label{eq:Ckmono}
	C_{k+1}=\frac{\eta_kQ_kC_k+E_{k+1}}{Q_{k+1}}
	\leq \frac{\eta_kQ_kC_k+C_k-\sigma \alpha_kt_k\|g_k\|^2}{Q_{k+1}}
	= C_k-\sigma\frac{\alpha_kt_k\|g_k\|^2}{Q_{k+1}},
\end{equation}
which means that $\{C_k\}$ is monotonically decreasing, though $\{E_k\}$ may not monotonically decrease in general. Actually, the monotonicity of $\{C_k\}$ in \eqref{eq:Ckmono} will play a key role in establishing the global convergence of Algorithm~\ref{alg:zhlmm}.

\begin{theorem}\label{thm:cvg-zhlmm}
	Suppose $E\in C^1(X,\R)$ and let $p$ be a peak selection of $E$ w.r.t. $L$. Further, $\{v_k\}\subset \mathcal{V}_0$ and $\{w_k=p(v_k)\}$ are sequences generated by Algorithm~\ref{alg:zhlmm}. Assume that (i) $p$ is continuous on $\mathcal{V}_0$; (ii) $t_k\geq\delta$ for some $\delta>0$, $k=0,1,\ldots$; and (iii) $\inf_{k\geq0}E_k>-\infty$ hold. Then
	\begin{itemize}
		\item[{\rm(a)}] $\sum_{k=0}^{\infty}\alpha_k\|g_k\|^2/Q_{k+1}<\infty$;
		\item[{\rm(b)}] if $\{w_k\}$ converges to some point $\bar{u}\in X$, $\bar{u}\notin L$ is a critical point.
	\end{itemize}
	Especially, if $\eta_{\max}<1$, then
	\begin{itemize}
		\item[{\rm(c)}] $\sum_{k=0}^{\infty}\alpha_k\|g_k\|^2<\infty$;
		\item[{\rm(d)}] every accumulation point of $\{w_k\}$ is a critical point not belonging to $L$;
		\item[{\rm(e)}] $\liminf_{k\to\infty}\|g_k\|=0$.
	\end{itemize}
	Further, if $E$ satisfies the (PS) condition, then
	\begin{itemize}
		\item[{\rm(f)}] $\{w_k\}$ contains a subsequence converging to a critical point $u_*\notin L$. In addition, if $u_*$ is isolated, $w_k\to u_*$ as $k\to\infty$.
	\end{itemize}
\end{theorem}
\begin{proof}
	Since $\{C_k\}$ is monotonically decreasing by \eqref{eq:Ckmono} and bounded from below by the assumption (iii) and Lemma~\ref{lem:ECA}, it converges to a finite number $C_*$. Then, \eqref{eq:Ckmono} and the assumption (ii) lead to the conclusion (a), i.e.,
	\[ \sum_{k=0}^{\infty}\frac{\alpha_k\|g_k\|^2}{Q_{k+1}}
	\leq \frac{1}{\sigma\delta}\sum_{k=0}^{\infty}(C_k-C_{k+1})
	= \frac{1}{\sigma\delta}(C_0-C_*)
	<\infty. \]
	
	Next, we verify the conclusion (b). By employing Lemma~\ref{lem:mono} and the assumption (ii), we have
	\[ \dist(w_k,L)=t_k\|v_k^\bot\|\geq\delta\|v_0^\bot\|>0,\quad k=0,1,\ldots.\]
	Thus, if $\{w_k\}$ converges to some point $\bar{u}\in X$, it implies that $\dist(\bar{u},L)\geq\delta\|v_0^\bot\|>0$. Immediately, we can obtain $\bar{u}\notin L$. In addition, Lemma~\ref{lem:home} indicates that $\{v_k\}$ converges to some $\bar{v}\in \mathcal{V}_0$ satisfying $\bar{u}=p(\bar{v})$.

	For the sake of contradiction, suppose that $\bar{u}$ is not a critical point, then $\nabla E(\bar{u})\neq0$. Since $E\in C^1(X,\R)$, one can obtain $g_k=\nabla E(w_k)\to\nabla E(\bar{u})\neq0$ as $k\to\infty$. Therefore, $\|g_k\|>\frac12\|\nabla E(\bar{u})\|>0$, for all $k$ large enough. Recalling the conclusion (a), it yields that $\sum_{k=0}^{\infty} \alpha_k/Q_{k+1}<\infty$. By utilizing \eqref{eq:qkex}, we can arrive at
	\begin{equation}\label{eq:sumakQ}
		\sum_{k=0}^{\infty}\frac{\alpha_k}{k+2}\leq\sum_{k=0}^{\infty}\frac{\alpha_k}{Q_{k+1}}<\infty.
	\end{equation}
	
	On the other hand, Lemma~\ref{lem:pv-tv} and the assumption (ii) lead to $t_{\bar{v}}=\lim_{k\to\infty}t_k$ $\geq\delta>0$. According to Lemma~\ref{lem:s0} and Lemma~\ref{lem:akgeq}, there exists $\underline{\alpha}>0$ such that, for all $k$ large enough,
	\begin{equation}\label{eq:akbt0}
		\alpha_k \geq \min\left\{\lambda_{\min},\rho\bar{\alpha}^A(v_k)\right\} \geq \min\left\{\lambda_{\min},\rho\underline{\alpha}\right\}>0.
	\end{equation}
	The combination of \eqref{eq:sumakQ} and \eqref{eq:akbt0} yields $\sum_{k=0}^\infty 1/(k+2)<\infty$, which is a contradiction. Thus, $\bar{u}\notin L$ is a critical point and the conclusion (b) is obtained.

	If $\eta_{\max}<1$, then revisiting \eqref{eq:qkex}, it is easy to see that
	\[ Q_{k+1}\leq 1+\sum_{j=0}^k \eta_{\max}^{j+1}<\frac{1}{1-\eta_{\max}}<\infty. \]
	Hence, the conclusion (c) directly follows from the conclusion (a).
	Moreover, by the conclusion (c) and an analogous argument in the proof of the conclusion (b), the conclusion (d) is obvious.

	Now, to consider the conclusion (e), suppose that $\delta_1:=\liminf_{k\to\infty}\|g_k\|>0$ by the contradiction argument. Then, $\|g_k\|\geq\delta_1/2>0$, for all $k$ large enough. One can see from the conclusion (c) that $\sum_{k=0}^{\infty}\alpha_k<\infty$ and $\sum_{k=0}^{\infty}\alpha_k\|g_k\|<\infty$. It immediately leads to $\alpha_k\to0$ as $k\to\infty$ and
	\[ \sum_{k=0}^{\infty}\|v_{k+1}-v_k\|\leq \sum_{k=0}^{\infty}\alpha_k\|g_k\|<\infty, \]
	where the inequality $\|v_{k+1}-v_k\|=\|v_k(\alpha_k)-v_k\|<\alpha_k\|g_k\|$ is employed according to Lemma~\ref{lem:vs}. Hence, $\{v_k\}$ is a Cauchy sequence. Note that $\{v_k\}$ is contained in the closed subset $\mathcal{V}_0$ which is complete, thus there exists $\bar{v}\in \mathcal{V}_0$ such that $v_k\to\bar{v}$ as $k\to\infty$. By the continuity of $p$ and $E'$, we have $g_k\to\nabla E(p(\bar{v}))$ as $k\to\infty$ and
	\[ \|\nabla E(p(\bar{v}))\|=\lim_{k\to\infty}\|g_k\|=\delta_1>0. \]
	However, from the conclusion (b), $p(\bar{v})=\lim_{k\to\infty}p(v_k)$ must be a critical point. This is a contradiction. Consequently, the conclusion (e) holds.

	The rest of the proof is to verify the conclusion (f). Due to the conclusion (e), one can find a subsequence $\{v_{k_i}\}$ such that $E'(w_{k_i})=E'(p(v_{k_i}))\to0$ in $X^*$ as $i\to\infty$. In view of Lemma~\ref{lem:ECA} and the assumption (iii), $\{E(w_{k_i})\}$ is bounded, i.e.,
	\[ \inf_{k\geq0}E_k\leq E(w_{k_i})\leq E_0,\quad i=0,1,\ldots. \]
	Then, by the (PS) condition, $\{w_{k_i}\}$ possesses a subsequence, still denoted by $\{w_{k_i}\}$, that converges to a critical point $u_*$, and $u_*\notin L$ in view of the conclusion (d). Finally, by using the assumption that $u_*$ is isolated and following the lines of the original proof for the global sequence convergence of the normalized Armijo-type LMM in \cite{Z2017CAMC} which is skipped here for brevity, we can reach the global sequence convergence, i.e., $w_k\to u_*$ as $k\to\infty$.
\end{proof}

\begin{remark}
	As discussed above, the normalized Armijo-type step-size search rule given in  \eqref{eq:ak-armijo} is a special case of the normalized ZH-type nonmonotone step-size search rule \eqref{eq:ZHcond} with $\eta_j=0$ ($j=0,1,\ldots,k$). Thus, Theorem~\ref{thm:cvg-zhlmm} covers the global convergence of the LMM with the normalized Armijo-type step-size search rule stated in \eqref{eq:ak-armijo}.
\end{remark}

\begin{remark}
	The assumption (ii) in Theorem~\ref{thm:cvg-zhlmm} is crucial to guarantee the critical point obtained by Algorithm~\ref{alg:zhlmm} is away from previously found critical points in $L$. When $L=\{0\}$, assumptions (i) and (ii) can be verified for energy functionals associated with several typical BVPs of PDEs occurring in Sect.~\ref{sec:numer}. Although it is not easy to theoretically prove the assumption (ii) in general cases, it is effective to numerically check it in practical computations.
\end{remark}

\section{Globally convergent BB-type LMM (GBBLMM)}
\label{sec:gbblmm}
In this section, we present the GBBLMM by using the nonmonotone globalizations developed in Sect.~\ref{sec:nmlmm} with a BB-type trial step-size for $\lambda_k$ at $v_k$. First, we modify the BB method in the optimization theory and construct the BB-type step-size for the LMM iteration.

\subsection{BB-type step-size for the LMM}
\label{sec:bbsslmm}

From Theorem~\ref{thm:lmt0}, under some assumptions, the local solution $v_*$ to the minimization problem
\begin{equation}\label{eq:lmm-minEpv}
	\min_{v\in\mathcal{V}_0}E(p(v))
\end{equation}
satisfies $\nabla E(p(v_*))=0$ and $p(v_*)\notin L$ (i.e., $p(v_*)\notin L$ is a critical point). As discussed in previous sections, the LMM iteration for solving the minimization problem \eqref{eq:lmm-minEpv} is
\begin{equation}\label{eq:lmm-iter}
	v_{k+1} =v_k(\alpha_k)= \frac{v_k-\alpha_kg_k}{\|v_k-\alpha_kg_k\|},\quad k=0,1,\ldots,
\end{equation}
where $g_k=\nabla E(p(v_k))$ and $v_k\in \mathcal{V}_0\subset S\backslash L$. A direct calculation shows that
\begin{align*}
	\left\|v_k(\alpha) - \left(v_k-\alpha g_k\right)\right\|
	&= \left|\frac{1}{\sqrt{1+\alpha^2\|g_k\|^2}}-1\right|\left\|v_k-\alpha g_k\right\|
	= \frac{\alpha^2\|g_k\|^2}{1+\sqrt{1+\alpha^2\|g_k\|^2}},
\end{align*}
and then $v_k(\alpha) = v_k-\alpha g_k + O\left(\alpha^2\|g_k\|^2\right)$. Hence, the linearized iterative scheme
\begin{equation}\label{lmm-iter-linearized}
	v_{k+1}=v_k-\alpha_k g_k\quad\mbox{or}\quad v_{k+1}=v_k-D_k\nabla E(p(v_k)), \quad k=0,1,\ldots,
\end{equation}
with $D_k=\alpha_kI$ and $I$ the identity operator in $X$, is a second-order approximation to the nonlinear iterative scheme \eqref{eq:lmm-iter}.

Similar to the BB method in the optimization theory, one can construct a linear iterative scheme \eqref{lmm-iter-linearized} for the nonlinear equation $\nabla E(p(v))=0$ with $\alpha_k$ as a BB-type step-size. Intuitively, such an $\alpha_k$ can serve as the step-size of the nonlinear iterative scheme \eqref{eq:lmm-iter} and is still called a BB-type step-size due to the fact that \eqref{lmm-iter-linearized} is a second-order approximation to \eqref{eq:lmm-iter}. For this purpose, the step-size $\alpha_k$ is chosen such that $D_k=\alpha_kI$ approximately satisfies the ``secant equation"
\begin{equation}\label{eq:dys}
	D_k y_k=s_k, \quad k=1,2\ldots,
\end{equation}
with $s_k=v_{k}-v_{k-1}$ and $y_k=g_{k}-g_{k-1}$. Then, solving the least-squares problem
\begin{equation}\label{eq:min-say}
	\min_{\alpha}\|s_k-\alpha y_k\|^2 \quad \mbox{or}\quad \min_{\beta}\|\beta s_k-y_k\|^2\quad\mbox{(w.r.t. $\beta=\alpha^{-1}$)},\quad k=1,2\ldots,
\end{equation}
yields BB-type step-sizes respectively as
\begin{equation}\label{eq:bbss}
	\alpha_k^{\text{BB1}}=\frac{(s_k, y_k)}{(y_k, y_k)} \quad\mbox{or}\quad
	\alpha_k^{\text{BB2}}=\frac{(s_k, s_k)}{(s_k, y_k)}\quad \mbox{if } (s_k, y_k)>0,\quad k=1,2\ldots.
\end{equation}

Another slightly different construction of the BB-type step-size can be obtained by considering the constrained minimization problem \eqref{eq:LMM2} from the point of view of manifold optimization. In fact, a Riemannian BB method for optimization on finite-dimensional Riemannian manifolds has been recently developed in \cite{IP2018IMAJNA} and the so-called vector transport is utilized to move vectors from a tangent space to another in it. Relevant to similar ideas in \cite{IP2018IMAJNA}, we construct the projected BB-type step-size for the LMM in Hilbert space. However, due to the unit sphere $S$ involved is a simple Hilbert-Riemannian manifold with a natural Riemannian metric induced by the inner product $(\cdot,\cdot)$ of $X$, we will avoid the general mathematical setting of infinite-dimensional Hilbert-Riemannian manifolds, which can be found, e.g., in \cite{Lang1995}.

The tangent space to the unit spherical manifold $S$ at a point $v\in S$ is given by $T_vS:=\{w\in X:(v, w)=0\}$, which is a Hilbert subspace equipped with the inner product $(\cdot,\cdot)_v=(\cdot,\cdot)$ and the norm $\|\cdot\|_v=\|\cdot\|$. Note that the second-order Fr\'{e}chet-derivative of a smooth functional defined on $S$ at $v\in S$ is a linear mapping from $T_vS$ to $T_vS$ \cite{Lang1995}. To preserve this property, both vectors $s_k$ and $y_k$ appearing in the secant equation \eqref{eq:dys} should belong to $T_{v_k}S$. Replacing $s_k$ and $y_k$ in \eqref{eq:min-say} with $\hat{s}_k=P_{v_k}s_k$ and $\hat{y}_k=P_{v_k}y_k$, respectively, where $P_{v_k}$ denotes the orthogonal projection from $X$ onto $T_{v_k}S$ with $P_{v_k}u=u-(u, v_k)v_k$, $\forall\,u\in X$, we can obtain the following projected BB-type step-size as
\begin{equation}\label{eq:pbbss}
	\alpha_k^{\text{PBB1}}=\frac{(\hat{s}_k,\hat{y}_k)}{(\hat{y}_k,\hat{y}_k)} \quad\mbox{or}\quad
	\alpha_k^{\text{PBB2}}=\frac{(\hat{s}_k,\hat{s}_k)}{(\hat{s}_k,\hat{y}_k)}\quad \mbox{if } (\hat{s}_k,\hat{y}_k)>0,\quad k=1,2\ldots.
\end{equation}

Clearly, $P_{v_k}(v_k)=0$. Applying the fact that $g_k\in[L,v_k]^\bot$ from Lemma~\ref{lem:orth}, it yields $P_{v_k}(g_k)=g_k$. Hence, for $k=1,2\ldots$, we have
\begin{align*}
	\hat{s}_k = -P_{v_{k}}\left(v_{k-1}\right) = -\alpha_{k-1}P_{v_k}(g_{k-1}),\quad
	\hat{y}_k = g_k-P_{v_k}(g_{k-1}) = g_k+\hat{s}_k/\alpha_{k}.
\end{align*}
We remark here that in \eqref{eq:bbss} and \eqref{eq:pbbss}, the main computational cost is the calculation of inner products. In addition, from the expression of $\hat{s}_k$ and $\hat{y}_k$, the key ingredient is the computation of the projection $P_{v_k}(g_{k-1})$ with $P_{v_k}(g_{k-1})=g_{k-1}-(g_{k-1},v_k)v_k$. In practice, compared to the BB-type step-size  \eqref{eq:bbss}, only one additional inner product, i.e., $(g_{k-1},v_k)$, needs to be calculated in the projected BB-type step-size \eqref{eq:pbbss}.

\subsection{BB-type LMM with nonmonotone globalizations}

Owing to essentially nonmonotone behaviors of the BB method, the nonlinearity and nonconvexity of the functional $E$ and the multiplicity and instability of saddle points, the convergence analysis for the LMM with the (projected) BB-type step-size is potentially difficult. To obtain a convergence safeguard, one needs to develop a globalization strategy. For this purpose, we propose the GBBLMM combined with a nonmonotone search strategy developed in Sect.~\ref{sec:nmlmm} with the trial step-size $\lambda_k$ determined by utilizing the BB-type step-size \eqref{eq:bbss} or the projected BB-type step-size \eqref{eq:pbbss}.

Since the (projected) BB-type step-size is defined for $k\geq1$, an appropriate initial trial step-size $\lambda_0$ is needed.
For $k\geq 1$, when $(s_k, y_k)\leq0$ (respectively, $(\hat{s}_k,\hat{y}_k)\leq0$), BB-type step-size \eqref{eq:bbss} (respectively, the projected BB-type step-size \eqref{eq:pbbss}) is unavailable. In this case, we simply set the trial step-size as $\lambda_k=\lambda_0$. In terms of cases when the (projected) BB-type step-size is unacceptably large or small, we must assume that the trial step-size $\lambda_k$ satisfies the condition
\[0<\lambda_{\min}\leq\lambda_k\leq\lambda_{\max},\quad k=1,2,\ldots. \]
Here, $\lambda_{\min}$ is to prevent $\lambda_k$ from being too small while $\lambda_{\max}$ is to avoid the search along the curve $\{v_k(\alpha):\alpha>0\}$ going too far and to enhance the stability of the algorithm. Hence, for $k\geq1$, the trial step-size $\lambda_k$ can be defined as one of the following,
\begin{align}
	\lambda_k &=
	\begin{cases}
		\min\left\{\max\left\{\alpha_k^{\text{BB}}, \lambda_{\min}\right\}, \lambda_{\max}\right\}, & \mbox{if } (s_k, y_k)>0, \\
		\lambda_0, &\mbox{otherwise},
	\end{cases} \label{eq:lambdak-bb} \\
	\lambda_k &= \begin{cases}
		\min\left\{\max\left\{\alpha_k^{\text{PBB}}, \lambda_{\min}\right\}, \lambda_{\max}\right\}, & \mbox{if } (\hat{s}_k,\hat{y}_k)>0, \\
		\lambda_0, &\mbox{otherwise},
	\end{cases} \label{eq:lambdak-pbb}
\end{align}
with $\alpha_k^{\text{BB}}\in\{\alpha_k^{\text{BB1}},\alpha_k^{\text{BB2}}\}$ and $\alpha_k^{\text{PBB}}\in\{\alpha_k^{\text{PBB1}},\alpha_k^{\text{PBB2}}\}$. Alternatively, similar to adaptive BB methods in optimization theory in Euclidean spaces \cite{DHL2019COA,HDL2021SIOPT}, one can also adaptively choose $\alpha_k^{\text{BB}}$ in \eqref{eq:lambdak-bb} and $\alpha_k^{\text{PBB}}$ in \eqref{eq:lambdak-pbb}. Two adaptive strategies to compute $\alpha_k^{\text{BB}}$ in \eqref{eq:lambdak-bb} are described in Appendix~\ref{app:adapBB}.

Revisiting Algorithms~\ref{alg:lmm} and \ref{alg:zhlmm}, the main steps of the GBBLMM are summarized in Algorithm~\ref{alg:gbblmm}.

\begin{algorithm}[!ht]
	\caption{Algorithm of the GBBLMM.}
	\label{alg:gbblmm}\normalsize
	\begin{enumerate}[\bf Step~1.]
		\item Perform the same initialization as Algorithm~\ref{alg:lmm} and Algorithm~\ref{alg:zhlmm}. Take $\lambda_0\in[\lambda_{\min},\lambda_{\max}]$. Set $k:=0$. Compute $w_0=p(v_0)$ and $g_0=\nabla E(w_0)$. Repeat {\bf Steps~2-4} until the stopping criterion is satisfied (e.g., $\|\nabla E(w_k)\|\leq \varepsilon_{\mathrm{tol}}$ for a given tolerance $0<\varepsilon_{\mathrm{tol}}\ll 1$), then output $u_n=w_k$.
		\item Find $\alpha_k=\lambda_k\rho^{m_k}$ with $m_k$ the smallest nonnegative integer satisfying the normalized ZH-type nonmonotone step-size search rule \eqref{eq:ZHcond} as in Algorithm~\ref{alg:zhlmm}.
		\item Set $v_{k+1}=v_k(\alpha_k)$ and $w_{k+1}=p(v_k(\alpha_k))$, compute $g_{k+1}=\nabla E(w_{k+1})$, and update $k:=k+1$.
		\item Compute $\lambda_k$ according to \eqref{eq:lambdak-bb} or \eqref{eq:lambdak-pbb} and go to {\bf Step~2}.
	\end{enumerate}
\end{algorithm}

\vspace{-10pt}

\section{Numerical experiments}
\label{sec:numer}
In this section, we apply Algorithm~\ref{alg:gbblmm} to find multiple unstable solutions of several nonlinear BVPs with variational structure. We take parameters $\sigma=10^{-4}$, $\rho=0.2$, $\eta_k\equiv0.85$, $\lambda_{\min}=10^{-6}$, $\lambda_{\max}=10$, then set $\lambda_0=0.1$ and $\lambda_k$ $(k\geq1)$ defined in \eqref{eq:lambdak-bb} with $\alpha_k^{\text{BB}}=\alpha_k^{\text{BB1}}$ unless specified. We remark here that, unless specified, the numerical experiments in this paper are implemented with MATLAB (R2017b) under the PC with the Inter Core i5-4300M CPU (2.60GHz) and a 4.00GB RAM. 
In the code of Algorithm~\ref{alg:gbblmm}, the MATLAB subroutine {\ttfamily fminunc} is called to compute the local peak selection.

\subsection{Semilinear Dirichlet BVPs}
\label{sec:numer:slpde}
Consider the homogeneous Dirichlet BVP
\begin{equation}\label{eq:slpde}
	-\Delta u(\x)=f(\x,u(\x))\quad \mbox{in }\Omega, \qquad
	u(\x)=0\quad\mbox{on }\partial\Omega,
\end{equation}
where $\Omega$ is a bounded domain in $\R^d$ with a Lipschitz boundary $\partial\Omega$ and the function $f:\bar{\Omega}\times\R\to \R$ satisfies standard hypotheses ($f$1)-($f$4) \cite{LZ2001SISC,Rabinowitz1986} given below. We omit the variable $\x\in\bar{\Omega}\subset\R^d$ in the following unless specified.
\begin{enumerate}[($f$1)]
	\item $f(\x,\xi)$ is locally Lipschitz on $\bar{\Omega}\times\R$ and $f(\x,\xi)=o(|\xi|)$ as $\xi\to0$;
	\item there are constants $c_1,c_2>0$ such that $|f(\x,\xi)|\leq c_1+c_2|\xi|^s$, where $s$ satisfies $1<s<2^*-1$ with $2^*:=2d/(d-2)$ if $d>2$ and $2^*:=\infty$ if $d=1,2$;
	\item there are constants $\mu>2$ and $R>0$ such that for $|\xi|\geq R$, $0<\mu F(\x,\xi)\leq f(\x,\xi)\xi$, where $F(\x,u)=\int_0^uf(\x,\xi)d\xi$;
	\item $f(\x,\xi)/|\xi|$ is increasing w.r.t. $\xi$ on $\R\backslash\{0\}$.
\end{enumerate}

The energy functional associated to the BVP \eqref{eq:slpde} is
\[ E(u)=\int_{\Omega}\left(\frac12|\nabla u|^2-F(\x,u)\right) d\x,\quad u\in X, \]
where $X:=H_0^1(\Omega)$ is equipped with the inner product and norm as
\[ (u, v)=\int_{\Omega}\nabla u\cdot\nabla vd\x,\quad \|u\|=\sqrt{(u,u)},\quad\forall \,u, v\in X. \]
According to \cite{LZ2001SISC,Rabinowitz1986}, the following facts are true:
\begin{enumerate}[(1)]
	\item Under hypotheses $(f1)$-$(f3)$, $E\in C^1(X,\R)$ and satisfies the (PS) condition. Further, any critical point of $E$ is a weak solution and also a classical solution of the BVP \eqref{eq:slpde}. Moreover, $E$ has a mountain pass structure with a unique local minimizer, i.e., $u=0$. Therefore, for any finite-dimensional closed subspace $L$ of $X$, the peak mapping $P(v)$ of $E$ w.r.t. $L$ at each $v\in S$ is nonempty.
	\item Under hypotheses $(f1)$-$(f4)$, for any finite-dimensional closed subspace $L$  of $X$, the uniqueness of a local peak selection $p(v)$ of $E$ w.r.t. $L$ implies its continuity at $v$. For the case of $L=\{0\}$, there is only one peak selection $p(v)$ of $E$ w.r.t. $L$ at any $v\in S$ and then $p$ is continuous on $S$. Moreover, in this case, there exists a constant $\delta>0$ such that $\mathrm{dist}(p(v),L)=\|p(v)\|\geq\delta>0$, $\forall\,v\in S$.
\end{enumerate}

It is clear that functions of the form $f(\x,\xi)=|\xi|^{\gamma-1}\xi$ with $1<\gamma<2^*-1$ satisfy hypotheses $(f1)$-$(f4)$, and so do all positive linear combinations of such functions. A typical example is that $f(\x, u)=|\x|^\ell u^3$, which leads to the H\'{e}non equation as
\begin{equation}\label{eq:henon}
	-\Delta u=|\x|^\ell u^3 \quad \mbox{in }\Omega, \qquad
	u=0\quad \mbox{on }\partial\Omega,
\end{equation}
where $\ell$ is a nonnegative parameter. The equation \eqref{eq:henon} was introduced by H\'{e}non \cite{H1973AA} when he studied rotating stellar structures. If $\ell=0$, this equation is also called the Lane-Emden equation.

By \eqref{eq:gk} and a simple calculation, the gradient $g_k\in X=H_0^1(\Omega)$ of $E$ at an iterative point $w_k=p(v_k)$ can be expressed as $g_k=w_k-\phi_k$ with $\phi_k$ the weak solution to the linear BVP
\[
-\Delta \phi_k = f(\x,w_k)\quad \mbox{in }\Omega,\quad
\phi_k=0\quad\mbox{on }\partial\Omega.
\]
Thus, the main cost for computing the gradient $g_k$ is solving the Poisson equation. For $d=2$, in our MATLAB code, {\ttfamily assempde}, a finite element subroutine provided by the MATLAB PDE Toolbox, is implemented with 32768 triangular elements to handle this task. In addition, the initial ascent direction $v_0$ is taken as the normalization of the solution to the following Poisson equation
\begin{equation}\label{eq:poisson-v0}
	-\Delta \tilde{v}_0=\mathbf{1}_{\Omega_1}-\mathbf{1}_{\Omega_2}\quad\mbox{in }\Omega,\qquad \tilde{v}_0=0\quad\mbox{on }\partial\Omega,
\end{equation}
where $\mathbf{1}_.=\mathbf{1}_.(\x)$ is the indicator function and $\Omega_1,\Omega_2$ are two disjoint subdomains of $\Omega$ for controlling the convexity of $v_0$. The stopping criterion for all examples below is set as $\|g_k\|<10^{-5}$ and $\max_{\x\in\Omega}\big|\Delta w_k(\x)+f(\x,w_k(\x))\big|<5\times10^{-5}$.

We remark here that examples and profiles of all solutions in this subsection are only shown on a square in $\R^2$. However, our approach is also available and efficient for different domains such as a ball, dumbbell or other complex domains.

\def\sfigw{.19\textwidth}
\def\sfigh{.15\textwidth}
\begin{figure}[!t]
	\footnotesize
	\makebox[\sfigw]{$u_1$}
	\makebox[\sfigw]{$u_2$}
	\makebox[\sfigw]{$u_3$}
	\makebox[\sfigw]{$u_4$}
	\makebox[\sfigw]{$u_5$} \\
	\hspace*{2ex}
	\includegraphics[width=\sfigw,height=\sfigh]{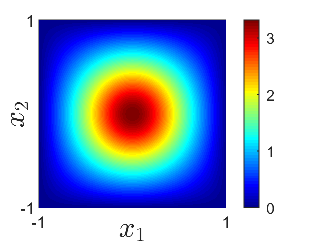}
	\includegraphics[width=\sfigw,height=\sfigh]{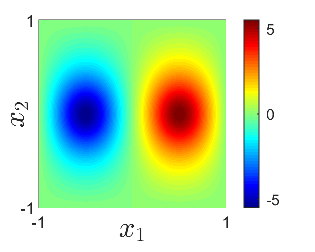}
	\includegraphics[width=\sfigw,height=\sfigh]{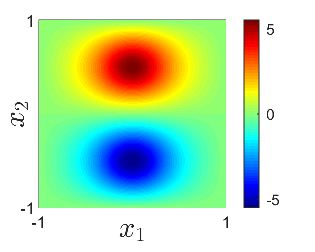}
	\includegraphics[width=\sfigw,height=\sfigh]{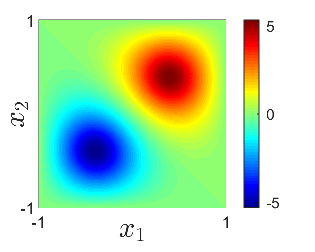}
	\includegraphics[width=\sfigw,height=\sfigh]{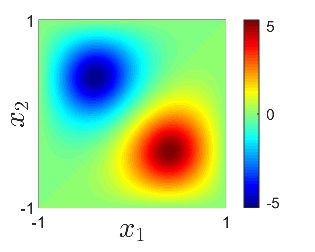} \\
	\makebox[\sfigw]{$u_6$}
	\makebox[\sfigw]{$u_7$}
	\makebox[\sfigw]{$u_8$}
	\makebox[\sfigw]{$u_9$}
	\makebox[\sfigw]{$u_{10}$} \\
	\hspace*{2ex}
	\includegraphics[width=\sfigw,height=\sfigh]{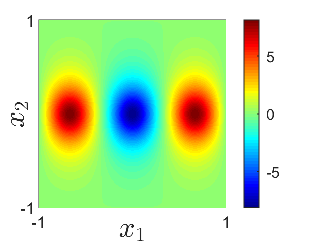}
	\includegraphics[width=\sfigw,height=\sfigh]{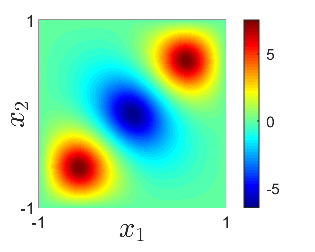}
	\includegraphics[width=\sfigw,height=\sfigh]{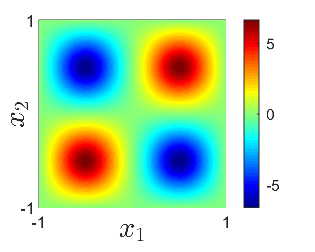}
	\includegraphics[width=\sfigw,height=\sfigh]{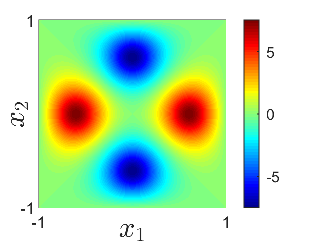}
	\includegraphics[width=\sfigw,height=\sfigh]{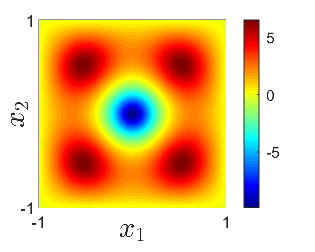} \\
	\vspace{-10pt}
	\caption{Profiles of ten solutions of the Lane-Emden equation on $\Omega=(-1,1)^2$.}
	\label{fig:LE10sols}
\end{figure}
\begin{table}[!t]
	\centering
	\small
	\caption{The initial information and energy functional value for each solution in Fig.~\ref{fig:LE10sols}.}
	\label{tab:lesq-init}
	\begin{tabular}{|c|l|l|l|r|}
		\hline
		$u_n$ & \quad$L$ & \quad$\Omega_1$ & ~~$\Omega_2$ & $E(u_n)$~ \\  \hline
		$u_1$ & $\{0\}$ & $\Omega$ & $\varnothing$ & 9.4460 \\  \hline
		$u_2$ & $[u_1]$ & $\Omega\cap\{x_1>0\}$ & $\Omega\backslash\Omega_1$ & 53.6731 \\  \hline
		$u_3$ & $[u_1]$ & $\Omega\cap\{x_2>0\}$ & $\Omega\backslash\Omega_1$ & 53.6731 \\  \hline
		$u_4$ & $[u_1]$ & $\Omega\cap\{x_1+x_2>0\}$ & $\Omega\backslash\Omega_1$ & 48.8807 \\  \hline
		$u_5$ & $[u_1]$ & $\Omega\cap\{x_1-x_2>0\}$ & $\Omega\backslash\Omega_1$ & 48.8807 \\  \hline
		$u_6$ & $[u_1,u_2]$ & $\Omega\cap\{|x_1|>0.2\}$ & $\Omega\backslash\Omega_1$ & 178.0269 \\  \hline
		$u_7$ & $[u_1,u_4]$ & $\Omega\cap\{|x_1+x_2|>0.3\}$ & $\Omega\backslash\Omega_1$ & 135.6335 \\  \hline
		$u_8$ & $[u_1,u_2,u_3]$ & $\Omega\cap\{x_1x_2>0\}$ & $\Omega\backslash\Omega_1$ & 151.3864 \\  \hline
		$u_9$ & $[u_1,u_4,u_5]$ & $\Omega\cap\{|x_1|>|x_2|\}$ & $\Omega\backslash\Omega_1$ & 195.7620 \\  \hline
		$u_{10}$ & $[u_1,u_2,u_3,u_8]$ & $\Omega\cap\{x_1^2+x_2^2>0.25\}$ & $\Omega\backslash\Omega_1$ & 233.9289 \\  \hline
	\end{tabular}
\end{table}
\begin{example}[Lane-Emden equation]\rm\label{ex:lesq}
	In this example, we employ Algorithm~\ref{alg:gbblmm} to compute a few nontrivial solutions to the Lane-Emden equation on a square, i.e., \eqref{eq:henon} with $\ell=0$, $d=2$ and $\x=(x_1,x_2)\in\Omega=(-1,1)^2$. Limited by the length of the paper, we only profile ten solutions obtained and labeled as $u_1,u_2,\ldots,u_{10}$ in  Fig.~\ref{fig:LE10sols}. For each solution, the information of the corresponding support space $L$, initial ascent direction $v_0$ and its energy functional value are described in Table~\ref{tab:lesq-init} with the notation `$[\cdots]$' denoting the space spanned by functions inside it. It is observed that the solution $u_1$ is a nontrivial positive solution with lowest energy and others are sign-changing solutions with higher energy. In fact, according to Theorem~1 in \cite{L1994MM}, due to $\Omega$ is convex in $\R^2$, $u_1$ is actually the unique positive solution to the Lane-Emden equation. Moreover, the existence of $u_1$ has been proved by the mountain pass lemma in \cite{Rabinowitz1986} and it is called the least-energy solution or the ground state solution.
\end{example}
\begin{table}[!t]
	\centering
	\caption{Numerical comparisons of the GBBLMM with traditional LMMs in terms of the CPU time (in seconds) for computing those solutions in Fig.~\ref{fig:LE10sols} with the shortest time underlined.}
	\label{tab:lesq2-comp}
	\footnotesize
	\begin{tabular}{|l|rrrrrrrr|}
		\hline
		$u$ & Exact & Armijo & BB1~ & PBB1 & BB2~ & PBB2 & ABB~ & APBB \\
		\hline
		$u_1$ &    2.1096 &   3.8458 &   1.3353 &   1.2192 &   1.3089 &   \underline{1.1933} &   1.3728 &   1.2367 \\
		$u_2$ &   15.2261 &   3.2181 &   \underline{1.8661} &   1.8975 &   1.8775 &   2.5628 &   2.0697 &   2.2227 \\
		$u_3$ &   17.7360 &   4.7765 &   2.0224 &   \underline{1.9484} &   2.0085 &   2.3639 &   2.1205 &   2.4500 \\
		$u_4$ &   26.2401 &   4.4908 &   2.7760 &   2.6123 &   \underline{2.4365} &   4.5057 &   2.8080 &   2.9564 \\
		$u_5$ &   25.7855 &   4.1724 &   4.8276 &   2.5380 &   \underline{2.3958} &   3.7228 &   3.0434 &   4.9856 \\
		$u_6$ &   95.2901 &  18.2996 &   6.8527 &   7.0348 &   \underline{6.2970} &   8.1687 &   6.4826 &   6.8969 \\
		$u_7$ &  106.4799 &  15.1870 &   4.1001 &   \underline{3.9299} &   5.0760 &  10.5012 &   5.1192 &   7.2954 \\
		$u_8$ &   53.2211 &  12.0485 &   4.2981 &   4.2186 &   4.0111 &   2.9897 &   4.0687 &   \underline{2.7712} \\
		$u_9$ &   87.9478 &   9.6975 &   5.4521 &   \underline{3.7206} &   6.2267 &   6.1504 &   6.8100 &   5.2849 \\
		$u_{10}$ &  174.1314 &  19.8673 &   9.8359 &   \underline{8.0634} &   8.9516 &   8.1358 &  11.0005 &   9.1701 \\
		\hline
	\end{tabular}
\end{table}

\def\sfigw{.38\textwidth}
\def\sfigh{.16\textheight}
\begin{figure}[!t]
	\centering
	\footnotesize
	\makebox[0.03\textwidth][r]{(a)}\includegraphics[width=\sfigw,height=\sfigh]{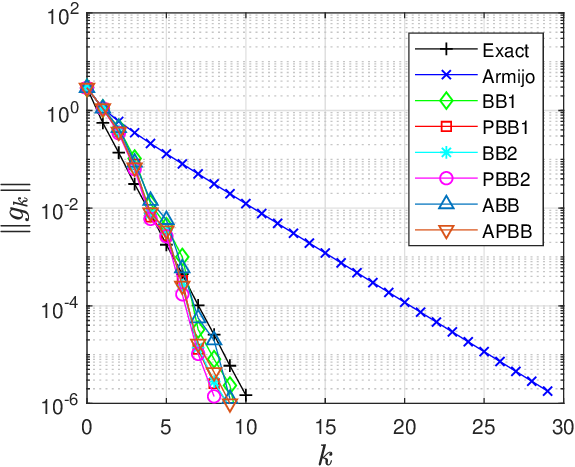} \qquad
	\makebox[0.03\textwidth][r]{(b)}\includegraphics[width=\sfigw,height=\sfigh]{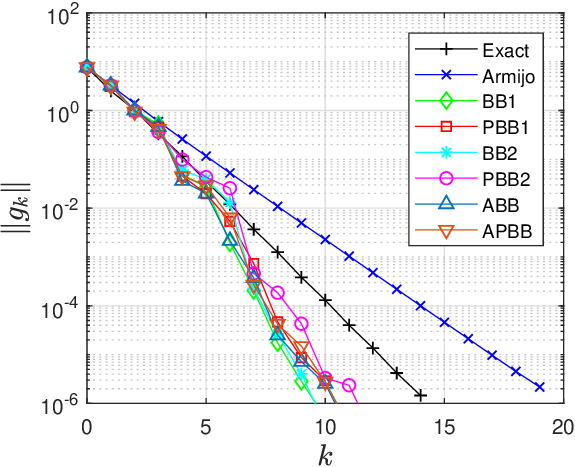} \\
	\makebox[0.03\textwidth][r]{(c)}\includegraphics[width=\sfigw,height=\sfigh]{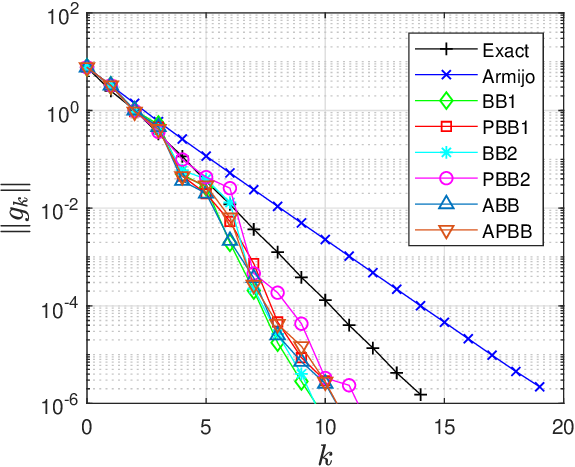} \qquad
	\makebox[0.03\textwidth][r]{(d)}\includegraphics[width=\sfigw,height=\sfigh]{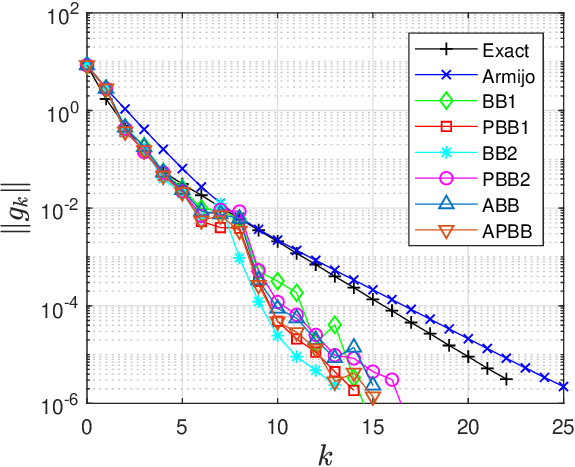} \\
	\makebox[0.03\textwidth][r]{(e)}\includegraphics[width=\sfigw,height=\sfigh]{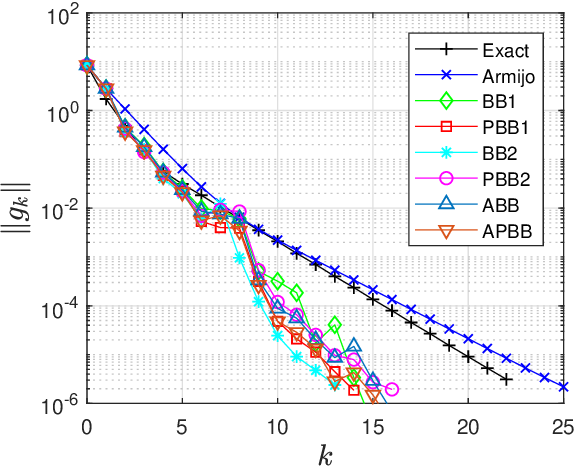} \qquad
	\makebox[0.03\textwidth][r]{(f)}\includegraphics[width=\sfigw,height=\sfigh]{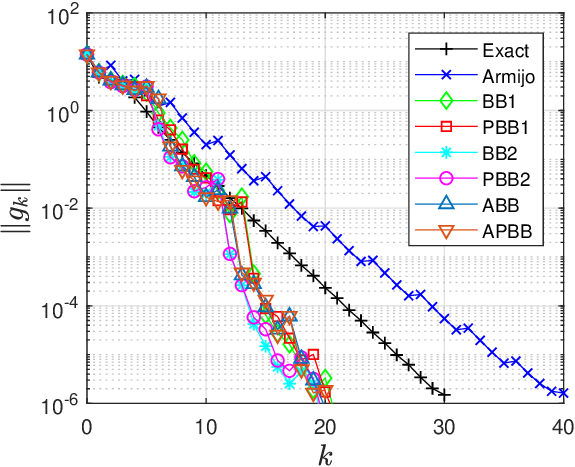} \\
	\makebox[0.03\textwidth][r]{(g)}\includegraphics[width=\sfigw,height=\sfigh]{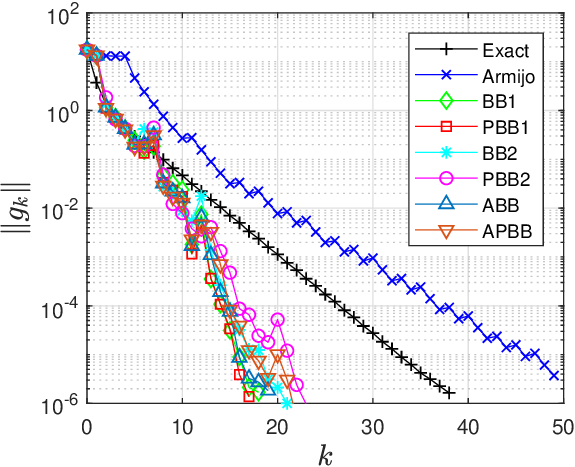} \qquad
	\makebox[0.03\textwidth][r]{(h)}\includegraphics[width=\sfigw,height=\sfigh]{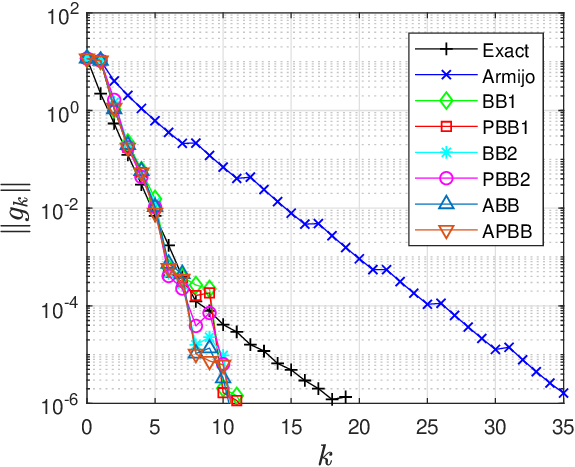} \\
	\makebox[0.03\textwidth][r]{(i)}\includegraphics[width=\sfigw,height=\sfigh]{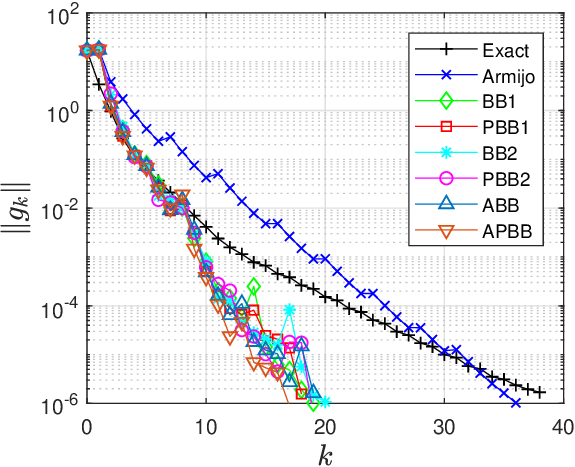} \qquad
	\makebox[0.03\textwidth][r]{(j)}\includegraphics[width=\sfigw,height=\sfigh]{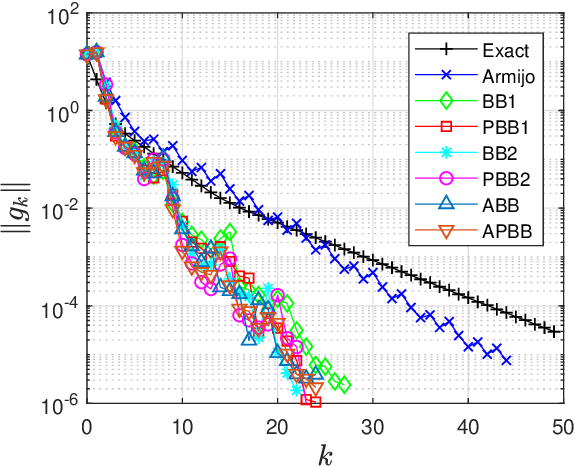} \\
	\caption{Numerical comparison of the GBBLMM with traditional LMMs in terms of the convergence rate for computing solutions in Fig.~\ref{fig:LE10sols}: (a) $\sim$ (j) for $u_1\sim u_{10}$, respectively. The horizontal and vertical coordinates represent the number of iterations and the norm of the gradient, respectively.}
	\label{fig:LE10sols-cvg}
\end{figure}

\begin{figure}[!t]
	\centering
	\footnotesize
	\includegraphics[width=.45\textwidth,height=.19\textheight]{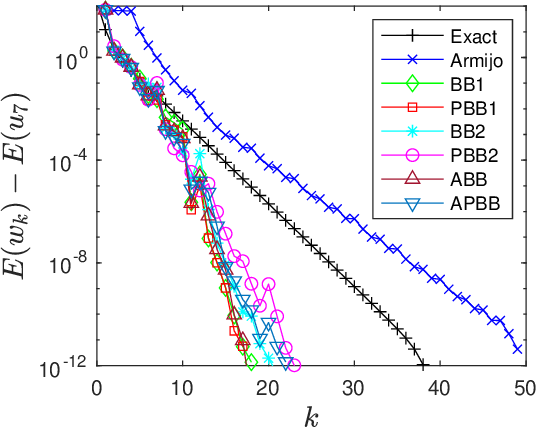} \qquad
	\includegraphics[width=.45\textwidth,height=.19\textheight]{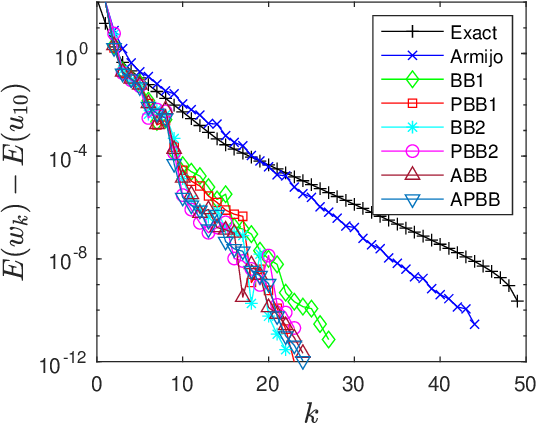} \\
	\caption{The changes of the relative energy functional values $E(w_k)-E(u)$ (in logarithmic scale) with respect to the number of iterations $k$ for the GBBLMM and traditional LMMs for computing the solutions $u=u_7$ (left) and $u=u_{10}$ (right) in Fig.~\ref{fig:LE10sols}.}
	\label{fig:LEu7u10-Ek}
\end{figure}

Then, we compare the efficiency of our GBBLMM with traditional LMMs for computing these nontrivial solutions $u_1,u_2,\ldots,u_{10}$ in Fig.~\ref{fig:LE10sols} with the same initial information stated in Table~\ref{tab:lesq-init}. The cost of the CPU time and number of iterations are exhibited in Table~\ref{tab:lesq2-comp} and Fig.~\ref{fig:LE10sols-cvg}, in which, respectively, `Exact' denotes Algorithm~\ref{alg:lmm} with the exact step-size search rule, i.e., the step-size $\alpha_k$ is chosen such that
\begin{equation*}
E(p(v_k(\alpha_k)))=\min_{0<\alpha\leq\lambda_{\max}}E(p(v_k(\alpha)));
\end{equation*}
`Armijo' denotes Algorithm~\ref{alg:lmm} with the normalized Armijo-type step-size search rule given in \eqref{eq:ak-armijo}. `BB1' (or `BB2') denotes Algorithm~\ref{alg:gbblmm} with $\lambda_k$ ($k\geq1$) defined in \eqref{eq:lambdak-bb} with $\alpha_k^{\text{BB}}=\alpha_k^{\text{BB1}}$ (or $\alpha_k^{\text{BB}}=\alpha_k^{\text{BB2}}$). `PBB1' (or `PBB2') denotes Algorithm~\ref{alg:gbblmm} with $\lambda_k$ ($k\geq1$) defined in \eqref{eq:lambdak-pbb} with $\alpha_k^{\text{PBB}}=\alpha_k^{\text{PBB1}}$ (or $\alpha_k^{\text{PBB}}=\alpha_k^{\text{PBB2}}$). `ABB' denotes Algorithm~\ref{alg:gbblmm} with $\lambda_k$ ($k\geq1$) defined in \eqref{eq:lambdak-bb} with $\alpha_k^{\text{BB}}=\alpha_k^{\text{BB1}}$ if $k$ is odd and $\alpha_k^{\text{BB}}=\alpha_k^{\text{BB2}}$ if $k$ is even. `APBB' denotes Algorithm~\ref{alg:gbblmm} with $\lambda_k$ ($k\geq1$) defined in \eqref{eq:lambdak-pbb} with $\alpha_k^{\text{BB}}=\alpha_k^{\text{PBB1}}$ if $k$ is odd and $\alpha_k^{\text{PBB}}=\alpha_k^{\text{PBB2}}$ if $k$ is even. Further, Fig.~\ref{fig:LEu7u10-Ek} plots the changes of energy functional values during the iterations of different algorithms for computing the solutions $u_7$ and $u_{10}$ (results for other solutions are similar and omitted here for brevity).

From Table~\ref{tab:lesq2-comp}, Figs.~\ref{fig:LE10sols-cvg}-\ref{fig:LEu7u10-Ek} and additional results not shown here, it is observed that our GBBLMM is quite efficient with less iterations and CPU time for solving the Lane-Emden equation, compared with the LMM using the exact step-size search rule or normalized Armijo-type step-size search rule. Moreover, for different choices of BB-type step-sizes, the corresponding algorithms of the GBBLMM have the similar efficiency. In addition, as shown in Fig.~\ref{fig:LEu7u10-Ek}, the energy functional values decrease monotonically during the iterations of `Armijo' and `Exact', with occasional growths for some of the GBBLMM algorithms.

Finally, motivated by the adaptive BB methods in optimization theory in Euclidean spaces \cite{DHL2019COA,HDL2021SIOPT}, we test the GBBLMM (Algorithm~\ref{alg:gbblmm}) with two adaptive strategies of BB-type step-sizes described in Appendix~\ref{app:adapBB}, denoted as `Adap1' and `Adap2'. In order to confirm that the BB-type step-size is the principal ingredient for the performance improvement against the normalized Armijo-type LMM method (i.e., `Armijo' in Table~\ref{tab:lesq2-comp}), we also test the nonmonotone ZH-type LMM (Algorithm~\ref{alg:zhlmm}) with the constant trial step-size $\lambda_k\equiv\lambda=0.1$ (as the same in `Armijo'), denoted as `ZH($\lambda$)'. The numerical comparisons of different LMMs in terms of the number of iterations and CPU time (in seconds) for computing the solutions $u_1\sim u_5$ in Fig.~\ref{fig:LE10sols} are presented in Table~\ref{tab:lesq2-comp-add} with `ABB' the same notation as in Table~\ref{tab:lesq2-comp}. The numerical results in Table~\ref{tab:lesq2-comp-add} further support our claims mentioned above and suggest that the adaptive strategies of BB-type step-sizes have the potential to further improve the efficiency of the GBBLMM.

\begin{table}[!t]
\centering
\caption{Numerical comparisons of different nonmonotone LMMs with the normalized Armijo-type LMM in terms of the number of iterations and CPU time (in seconds) for computing the first five solutions in Fig.~\ref{fig:LE10sols}. (The numerical experiments in this table are implemented with MATLAB (R2020b) under the PC with the Intel Core i7 CPU (2.7 GHz) and a 8.00 GB RAM.)}
	\label{tab:lesq2-comp-add}
\footnotesize
\begin{tabular}{|c|cr|cr|cr|cr|cr|}
	\hline
	\multirow{2}{*}{$u$} 
	& \multicolumn{2}{|c|}{ABB} & \multicolumn{2}{|c|}{Adap1} & \multicolumn{2}{|c|}{Adap2} 
	& \multicolumn{2}{|c|}{ZH($\lambda$)} & \multicolumn{2}{|c|}{Armijo} \\
	\cline{2-11}
	& \#its & time & \#its & time & \#its & time & \#its & time & \#its & time \\
	\hline
	$u_{1}$  &  9 & 0.58  & 10 & 0.73  &  9 & 0.62  & 29 & 1.80  & 29 & 1.87 \\
	$u_{2}$  & 11 & 0.93  & 10 & 0.92  & 10 & 0.91  & 19 & 1.49  & 19 & 1.43 \\
	$u_{3}$  & 11 & 0.92  & 10 & 0.82  & 10 & 0.93  & 19 & 1.49  & 19 & 1.44 \\
	$u_{4}$  & 15 & 1.23  & 15 & 1.20  & 13 & 1.11  & 25 & 1.91  & 25 & 1.72 \\
	$u_{5}$  & 15 & 1.22  & 15 & 1.20  & 13 & 1.17  & 25 & 1.92  & 25 & 1.72 \\
	\hline
\end{tabular}
\end{table}

\begin{example}[H\'{e}non equation]\rm\label{ex:hensq}
	Now, we employ Algorithm~\ref{alg:gbblmm} to compute a few nontrivial solutions to the H\'{e}non equation \eqref{eq:henon} on $\Omega=(-1,1)^2$. First, for different $\ell\geq0$, we compute the ground state solution by taking $L=\{0\}$ and $v_0$ according to \eqref{eq:poisson-v0} with $\Omega_1=\{\x=(x_1,x_2)\in\Omega:x_1>0,x_2>0\}$ and $\Omega_2=\varnothing$. The profiles of corresponding ground state solutions with different $\ell$ are presented in Fig.~\ref{fig:hensq-gs}.

	From Fig.~\ref{fig:hensq-gs} and other numerical results in various domain not shown here, one can numerically observe that, when $\ell$ is close to zero (approximately, $\ell\leq0.5$), the ground state solution is symmetric and attains its maximum value at the center of the domain; when $\ell$ is large (approximately, $\ell\geq0.6$), the maximizer of the ground state solution is gradually away from the center of the domain, i.e., the symmetry-breaking occurs. This similar interesting phenomenon for the H\'{e}non equation on the unit ball was first numerically observed in \cite{CZN2000IJBC} and then theoretically verified in \cite{SWS2002CCM}. To the best of our knowledge, the rigorous analysis for the exact critical value of $\ell$ that determines whether symmetry-breaking occurs for the H\'{e}non equation on other domains besides the unit ball is still an open problem. It should be one of interesting issues considered in our future work.
\end{example}

\def\sfigw{.19\textwidth}
\def\sfigh{.15\textwidth}
\begin{figure}[!t]
	\footnotesize
	\makebox[\sfigw]{$\ell=0$}
	\makebox[\sfigw]{$\ell=0.1$}
	\makebox[\sfigw]{$\ell=0.2$}
	\makebox[\sfigw]{$\ell=0.3$}
	\makebox[\sfigw]{$\ell=0.4$} \\
	\hspace*{2ex}
	\includegraphics[width=\sfigw,height=\sfigh]{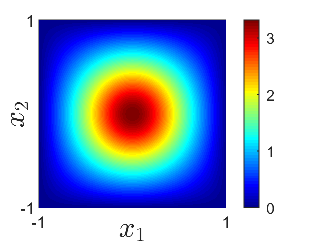}
	\includegraphics[width=\sfigw,height=\sfigh]{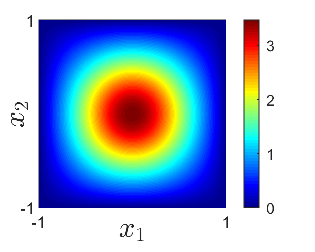}
	\includegraphics[width=\sfigw,height=\sfigh]{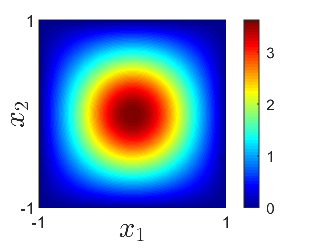}
	\includegraphics[width=\sfigw,height=\sfigh]{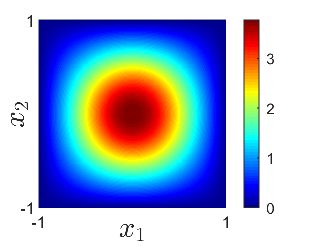}
	\includegraphics[width=\sfigw,height=\sfigh]{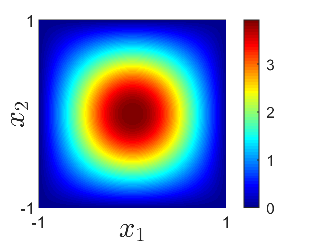} \\
	\makebox[\sfigw]{$\ell=0.5$}
	\makebox[\sfigw]{$\ell=0.6$}
	\makebox[\sfigw]{$\ell=0.7$}
	\makebox[\sfigw]{$\ell=0.8$}
	\makebox[\sfigw]{$\ell=0.9$} \\
	\hspace*{2ex}
	\includegraphics[width=\sfigw,height=\sfigh]{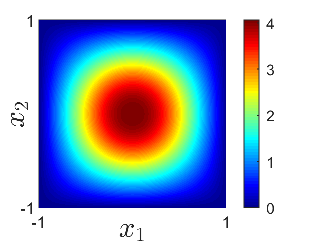}
	\includegraphics[width=\sfigw,height=\sfigh]{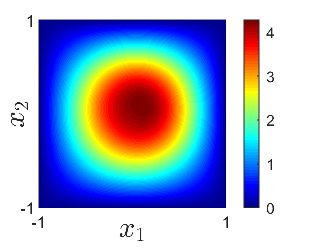}
	\includegraphics[width=\sfigw,height=\sfigh]{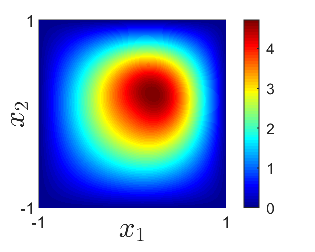}
	\includegraphics[width=\sfigw,height=\sfigh]{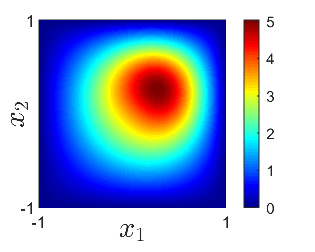}
	\includegraphics[width=\sfigw,height=\sfigh]{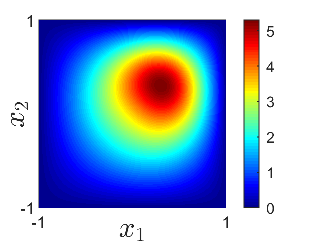} \\
	\makebox[\sfigw]{$\ell=1$}
	\makebox[\sfigw]{$\ell=2$}
	\makebox[\sfigw]{$\ell=3$}
	\makebox[\sfigw]{$\ell=4$}
	\makebox[\sfigw]{$\ell=5$} \\
	\hspace*{2ex}
	\includegraphics[width=\sfigw,height=\sfigh]{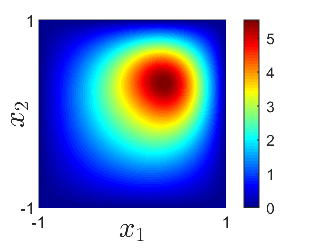}
	\includegraphics[width=\sfigw,height=\sfigh]{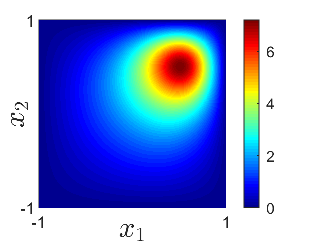}
	\includegraphics[width=\sfigw,height=\sfigh]{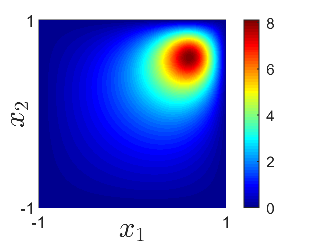}
	\includegraphics[width=\sfigw,height=\sfigh]{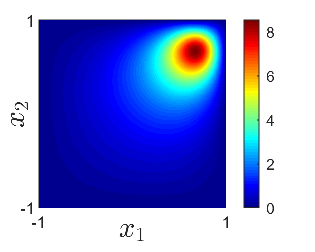}
	\includegraphics[width=\sfigw,height=\sfigh]{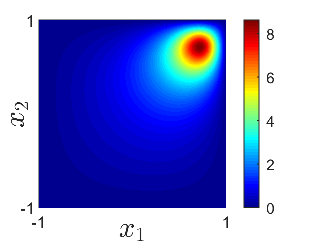} \\
	\vspace{-10pt}
	\caption{Profiles of ground state solutions of the H\'{e}non equation on $\Omega=(-1,1)^2$ with different $\ell s$.}
	\label{fig:hensq-gs}
\end{figure}
\def\sfigw{.24\textwidth}
\def\sfigh{.172\textwidth}
\begin{figure}[!t]
	\footnotesize
	\makebox[\sfigw]{$u_1$}
	\makebox[\sfigw]{$u_2$}
	\makebox[\sfigw]{$u_3$}
	\makebox[\sfigw]{$u_4$} \\
	\hspace*{2ex}
	\includegraphics[width=\sfigw,height=\sfigh]{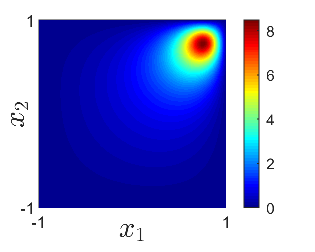}
	\includegraphics[width=\sfigw,height=\sfigh]{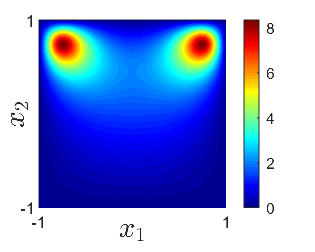}
	\includegraphics[width=\sfigw,height=\sfigh]{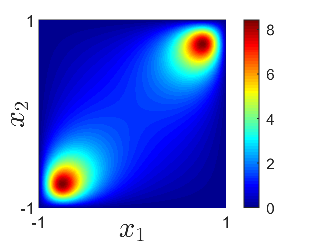}
	\includegraphics[width=\sfigw,height=\sfigh]{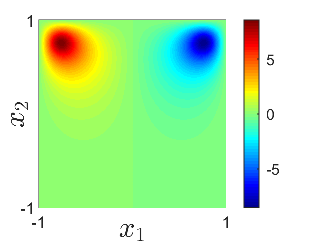} \\
	\makebox[\sfigw]{$u_5$}
	\makebox[\sfigw]{$u_6$}
	\makebox[\sfigw]{$u_7$}
	\makebox[\sfigw]{$u_8$} \\
	\hspace*{2ex}
	\includegraphics[width=\sfigw,height=\sfigh]{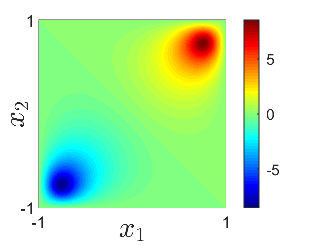}
	\includegraphics[width=\sfigw,height=\sfigh]{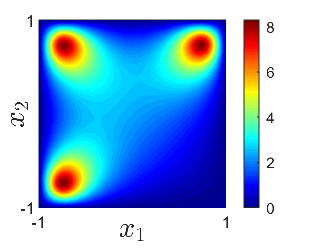}
	\includegraphics[width=\sfigw,height=\sfigh]{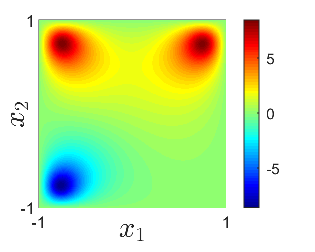}
	\includegraphics[width=\sfigw,height=\sfigh]{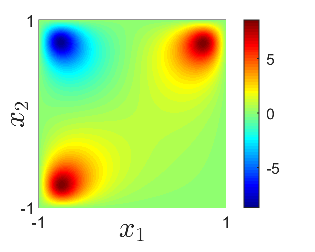} \\
	\makebox[\sfigw]{$u_9$}
	\makebox[\sfigw]{$u_{10}$}
	\makebox[\sfigw]{$u_{11}$}
	\makebox[\sfigw]{$u_{12}$} \\
	\hspace*{2ex}
	\includegraphics[width=\sfigw,height=\sfigh]{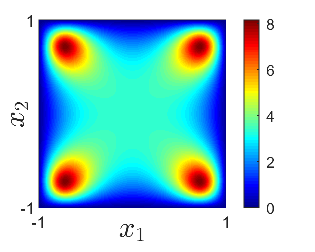}
	\includegraphics[width=\sfigw,height=\sfigh]{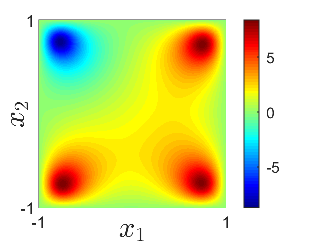}
	\includegraphics[width=\sfigw,height=\sfigh]{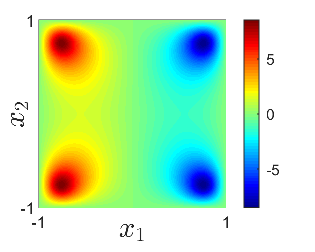}
	\includegraphics[width=\sfigw,height=\sfigh]{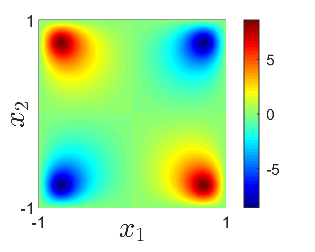} \\
	\vspace{-10pt}
	\caption{Profiles of twelve solutions of the H\'{e}non equation with $\ell=6$ on $\Omega=(-1,1)^2$.}
	\label{fig:hensq12sols}
\end{figure}
\begin{table}[!t]
	\centering
	\small
	\caption{The initial information and energy functional value for each solution in Fig.~\ref{fig:hensq12sols}.}
	\label{tab:hensq-init}
	\begin{tabular}{|c|l|l|l|r|}
		\hline
		$u_n$ & \quad$L$ & \quad$\Omega_1$ & \quad$\Omega_2$ & $E(u_n)$~ \\  \hline
		$u_1$ & $\{0\}$ & $\Omega\cap\{x_1>0,x_2>0\}$ & $\varnothing$ & 61.9634 \\  \hline
		$u_2$ & $[u_1]$ & $\Omega\cap\{x_1<0,x_2>0\}$ & $\varnothing$ & 120.7887 \\  \hline
		$u_3$ & $[u_1]$ & $\Omega\cap\{x_1<0,x_2<0\}$ & $\varnothing$ & 122.4078 \\  \hline
		$u_4$ & $[u_1]$ & $\Omega\cap\{x_2>0\}$ & $\varnothing$ & 126.6988 \\  \hline
		$u_5$ & $[u_1]$ & $\Omega\cap\{x_1>0,x_2>0\}$ & $\Omega\cap\{x_1<0,x_2<0\}$ & 125.3561 \\  \hline
		$u_6$ & $[u_1,u_2]$ & $\Omega\cap\{x_1<0,x_2<0\}$ & $\varnothing$ & 177.6068 \\  \hline
		$u_7$ & $[u_1,u_3]$ & $\Omega\cap\{x_2>0\}$ & $\varnothing$ & 187.1379 \\  \hline
		$u_8$ & $[u_1,u_4]$ & $\Omega\cap\{x_1<0,x_2<0\}$ & $\varnothing$ & 189.9406 \\  \hline
		$u_9$ & $[u_1,u_2,u_6]$ & $\Omega\cap\{x_1>0,x_2<0\}$ & $\varnothing$ & 230.0141 \\  \hline
		$u_{10}$ & $[u_1,u_2,u_6]$ & $\Omega\cap\{x_2<0\}$ & $\Omega\cap\{x_2>0\}$ & 247.0220 \\  \hline
		$u_{11}$ & $[u_1,u_2,u_6]$ & $\Omega\cap\{x_1x_2>0\}$ & $\Omega\cap\{x_1x_2<0\}$ & 250.6746 \\  \hline
		$u_{12}$ & $[u_1,u_2,u_6]$ & $\Omega\cap\{x_1x_2<0\}$ & $\varnothing$ & 255.9728 \\  \hline
	\end{tabular}
\end{table}

Then, taking $\ell=6$, we profile twelve solutions obtained and labeled as $u_1,u_2,\ldots,$ $u_{12}$ in Fig.~\ref{fig:hensq12sols}.  For each solution, the information of the corresponding support space $L$, initial ascent direction $v_0$ and its energy functional value is listed in Table~\ref{tab:hensq-init}. It is observed that $u_1$, $u_2$, $u_3$, $u_6$ and $u_9$ are five positive solutions and others are sign-changing solutions. Distinguished from the case of $\ell=0$ (see the Lane-Emden equation in Example~\ref{ex:lesq}), the positive solution is no longer unique and more nontrivial solutions spring up. The multiplicity of positive solutions for large $\ell$ is also numerically observed and theoretically analyzed in some literature; see, e.g., \cite{CZN2000IJBC,LZ2002SISC,SWS2002CCM,YLZ2008SCSA}. In addition, our approach is also compared with traditional LMMs for the H\'{e}non equation with the significant superiority in the performance, which is similar as that in Example~\ref{ex:lesq} and skipped here due to the limit of the length.

\subsection{Elliptic PDEs with nonlinear boundary conditions}
Consider the following BVP
\begin{equation}\label{eq:nlbc-model}
	-\Delta u+au=0 \quad \mbox{in }\Omega, \qquad
	\frac{\partial u}{\partial\mathbf{n}}=q(\x,u) \quad \mbox{on }\partial\Omega,
\end{equation}
where $\Omega\subset \R^d$ is a bounded open domain with a Lipschitz boundary $\partial\Omega$, the constant $a>0$, $\mathbf{n}=\mathbf{n}(\x)$ denotes the unit outward normal vector to $\partial\Omega$ at $\x$, and the nonlinear function $q(\x, \xi)$ satisfies the following regularity and growth hypotheses \cite{LWZ2013JSC}:
\begin{enumerate}[($q$1)]
	\item $q(\x,\xi)\in C^1(\partial\Omega\times\R,\R)$ and $q(\x,0)=\partial_{\xi}q(\x,\xi)|_{\xi=0}=0$, $\forall \,\x\in\partial\Omega$;
	\item there are constants $c_1,c_2>0$ such that $|q(\x,\xi)|\leq c_1+c_2|\xi|^s$, $\forall\, \x\in\partial\Omega$, where $s$ satisfies $1<s<d/(d-2)$ for $d>2$ and $1<s<\infty$ for $d=2$;
	\item there are constants $\mu>2$, $R>0$ such that $0\leq\mu Q(\x,\xi)\leq \xi q(\x,\xi)$, $\forall\,|\xi|>R$, $\x\in\partial\Omega$, where $Q(\x,u)=\int_0^uq(\x,\xi)d\xi$;
	\item $\partial_{\xi}q(\x,\xi)>q(\x,\xi)/\xi$, $\forall\, (\x,\xi)\in\partial\Omega\times(\R\backslash\{0\})$.
\end{enumerate}

It is worthwhile to point out that the BVP \eqref{eq:nlbc-model} appears in many scientific fields, such as corrosion/oxidation modeling, metal-insulator or metal-oxide semiconductor systems; see \cite{A2004NFAO,LWZ2013JSC} and references therein. However, there are few studies on the computation of multiple solutions to it.

Define the space
\[
X=\left\{u\in H^1(\Omega):\int_{\Omega}(\nabla u\cdot\nabla v+a uv)d\x=\int_{\partial\Omega}\frac{\partial u}{\partial\mathbf{n}}v ds, \; \forall\, v\in H^1(\Omega)\right\},
\]
which is equipped with the inner product and norm as
\begin{equation}\label{eq:nlbc-ipX}
	(u, v)=\int_{\partial\Omega}\frac{\partial u}{\partial\mathbf{n}}vds,\quad
	\|u\|=\sqrt{(u,u)},\quad \forall\, u, v\in X.
\end{equation}
According to \cite{A2004NFAO,LWZ2013JSC}, $X$ is a Hilbert space and $X=H^{\frac12}(\partial\Omega)$ in the sense of equivalent norms.
In addition, for the inner product in $H^1(\Omega)$ defined as $(u, v)_a=\int_{\Omega}(\nabla u\cdot\nabla v+a uv)d\x$, $\forall\, u, v\in H^1(\Omega)$, $X$ is the $(\cdot,\cdot)_a$-orthogonal complement of $H_0^1(\Omega)$ in $H^1(\Omega)$ and has an $(\cdot,\cdot)_a$-orthogonal basis formed by the Steklov eigenfunctions \cite{A2004NFAO,LWZ2013JSC}.

Clearly, $X$ contains all solutions of the BVP \eqref{eq:nlbc-model} in the weak sense and the energy functional associated to the BVP \eqref{eq:nlbc-model} for $u\in X$ can be written as
\[
E(u) =\frac12\int_{\Omega}\left(|\nabla u|^2+au^2\right)d\x- \int_{\partial\Omega}Q(\x,u)ds
= \int_{\partial\Omega}\left(\frac12\frac{\partial u}{\partial\mathbf{n}}u-Q(\x,u)\right)ds.
\]
Under hypotheses ($q$1) and ($q2$), $E\in C^2(X,\R)$ and satisfies the (PS) condition \cite{LWZ2013JSC,Rabinowitz1986}. If $q$ satisfies hypotheses ($q$1)-($q$3), then the BVP \eqref{eq:nlbc-model} has at least three nontrivial solutions \cite{LWZ2013JSC,W1991AIHPNLA}. If, in addition to hypotheses ($q$1)-($q$3), $q(\x,\xi)$ is odd in $\xi$, the existence of infinitely many solutions to the BVP \eqref{eq:nlbc-model} can be established by following the proof of Theorem~9.12 in \cite{Rabinowitz1986}. Under hypotheses ($q$1)-($q$4), when $L=\{0\}$, the peak selection $p(v)$ is uniquely defined for each $v\in S$ and is $C^1$ \cite{LWZ2013JSC}. Moreover, in this case, there exists a constant $\delta>0$ such that $\mathrm{dist}(p(v),L)=\|p(v)\|\geq\delta>0$, $\forall\,v\in S$; see Proposition~4 in \cite{LWZ2013JSC}.

By the definition of the inner product \eqref{eq:nlbc-ipX}, the gradient $g=\nabla E(u)\in X$ satisfies
\[
\int_{\partial\Omega}\frac{\partial g}{\partial\mathbf{n}}vds
= \langle E'(u),v\rangle
= \frac{d}{d\tau}E(u+\tau v)\Big|_{\tau=0}
= \int_{\partial\Omega}\left(\frac{\partial u}{\partial\mathbf{n}}-q(x,u)\right)vds, \;\; \forall\, v\in X.
\]
Recalling the definition of $X$, it implies that $g\in X$ is the weak solution to the linear elliptic BVP as
\begin{equation}\label{eq:nbc-g-pde}
	-\Delta g+ag=0 \quad \mbox{in }\Omega, \qquad
	\frac{\partial g}{\partial\mathbf{n}}=b \quad \mbox{on }\partial\Omega,
\end{equation}
with $b=b(\x)=\frac{\partial u}{\partial\mathbf{n}}(\x)-q(\x, u(\x))$, $\x\in\partial\Omega$.
In practice, the linear elliptic BVP \eqref{eq:nbc-g-pde} can be solved numerically by finite difference methods, finite element methods and so on. In our experiments, an efficient boundary element method (BEM) \cite{CZ1992,LWZ2013JSC} with 1024 boundary elements is applied to solve it.

We now employ our GBBLMM to solve for multiple solutions of the BVP \eqref{eq:nlbc-model} with $d=2$, $a=1$ and $q(\x, u)=u^3$ for two different domains stated in Example \ref{ex:nlbc_circ} and Example \ref{ex:nlbc_rect}, respectively. It is easy to see that hypotheses ($q$1)-($q$4) are satisfied for this case. In addition, for the following examples, the initial ascent direction $v_0$ is taken as the normalization of
\begin{equation}\label{eq:v0bie}
	\tilde{v}_0(\x)=\int_{\partial\Omega}\Phi(|\x-\y|)\rho_0(\y)ds_{\y},\quad \x\in \Omega,
\end{equation}
for some given function $\rho_0$ defined on $\partial\Omega$, where $\Phi$ is the fundamental solution to the linear elliptic operator $-\Delta+aI$ and given as
\begin{equation}\label{eq:fundsol}
	\Phi(|\x-\y|)= \frac{1}{2\pi}K_0\left(\sqrt{a}|\x-\y|\right),\quad \x, \y\in\Omega,
\end{equation}
with $K_0$ the modified Bessel function of the second kind of order $0$. The stopping criterion is set as $\|g_k\|<10^{-5}$ and $\max_{\x\in\partial\Omega}\left|\frac{\partial w_k}{\partial\mathbf{n}}(\x)-q(\x, w_k(\x))\right|<5\times10^{-5}$.

\def\sfigw{.18\textwidth}
\def\sfigh{.09\textwidth}
\begin{figure}[!t]
	\centering
	\footnotesize
	\makebox[\sfigw]{$u_1$}
	\makebox[\sfigw]{$u_2$}
	\makebox[\sfigw]{$u_3$}
	\makebox[\sfigw]{$u_4$}
	\makebox[\sfigw]{$u_5$} \\
	\quad
	\includegraphics[width=\sfigw]{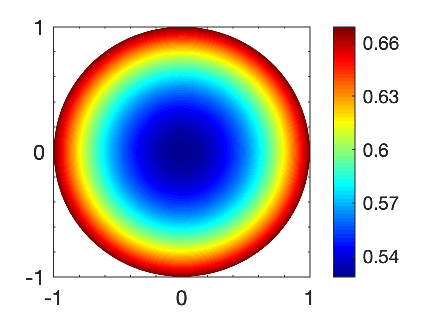}
	\includegraphics[width=\sfigw]{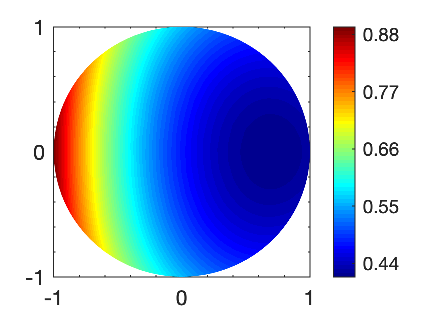}
	\includegraphics[width=\sfigw]{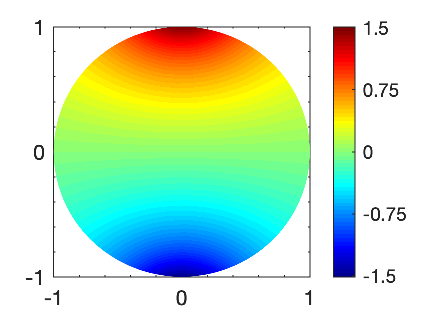}
	\includegraphics[width=\sfigw]{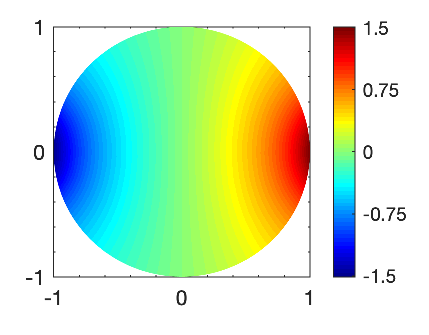}
	\includegraphics[width=\sfigw]{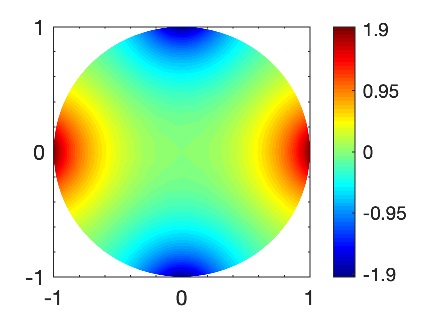} \\
	\resizebox*{\sfigw}{\sfigh}{\includegraphics{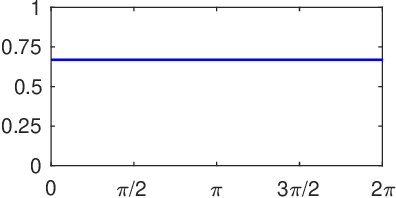}}
	\resizebox*{\sfigw}{\sfigh}{\includegraphics{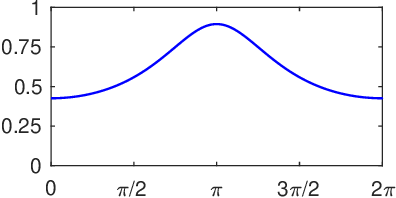}}
	\resizebox*{\sfigw}{\sfigh}{\includegraphics{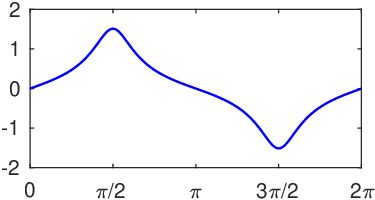}}
	\resizebox*{\sfigw}{\sfigh}{\includegraphics{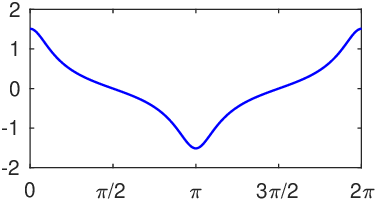}}
	\resizebox*{\sfigw}{\sfigh}{\includegraphics{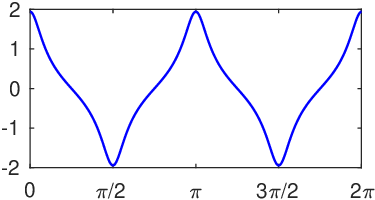}} \\
	\caption{Profiles of five nontrivial solutions $u_1\sim u_5$ in Example~\ref{ex:nlbc_circ} inside the domain (top row) and on the boundary (bottom row). In each subplot in the bottom panel, the horizontal axis represents the value of the boundary parameter $\theta$ ($0\leq\theta\leq 2\pi$) which corresponds to the boundary point $(x_1,x_2)=(\cos\theta,\sin\theta)$ as described in \eqref{ex:nlbc_circ_theta}.}
	\label{fig:nlbc_circ5sols}
\end{figure}

\begin{table}[!t]
	\centering
	\small
	\caption{The initial information and energy functional value for each solution in Example~\ref{ex:nlbc_circ}.}
	\label{tab:nlbc_circle}
	\begin{tabular}{|c|c|c|c|c|c|}
		\hline
		$u_n$ & $u_1$ & $u_2$ & $u_3$ & $u_4$ & $u_5$ \\  \hline
		$L$ & $\{0\}$ & $\{0\}$ & $[u_2]$ & $[u_2]$ & $[u_1,u_3,u_4]$ \\  \hline
		$\rho_0(\x(\theta))$ & $1$ & $1-\cos\theta$ & $\sin\theta$ & $\cos\theta$ & $\cos2\theta$ \\  \hline
		$E(u_n)$ & 0.3148 & 0.3105 & 1.3025 & 1.3025 & 4.1364 \\  \hline
	\end{tabular}
\end{table}

\begin{example}[A circle domain case]\rm\label{ex:nlbc_circ}
	Take the domain $\Omega=\{(x_1,x_2):x_1^2+x_2^2<1\}$ with its boundary $\partial\Omega=\{(x_1,x_2):x_1^2+x_2^2=1\}$ parametrized by
	\begin{equation}\label{ex:nlbc_circ_theta}
		\x=\x(\theta) = (x_1(\theta),x_2(\theta))=(\cos\theta,\sin\theta),\quad \theta\in[0,2\pi].
	\end{equation}
	We show five solutions $u_1,u_2,\ldots,u_5$ obtained in Fig.~\ref{fig:nlbc_circ5sols} with their profiles inside the domain and on the boundary. For each solution, the information of its corresponding support space $L$, initial ascent direction $v_0$ (determined by $\rho_0$ via \eqref{eq:v0bie}) and energy functional value is listed in Table~\ref{tab:nlbc_circle}. It is noted that the boundary value of the solution $u_1$ is a constant, approximately to $0.6691$. Actually, the algorithm for computing $u_1$ needs only one iteration. Compared the efficiency of our GBBLMM with that of traditional LMMs, our approach can be observed to perform much better with less iterations and CPU time for the BVP \eqref{eq:nlbc-model} on a circle domain. The relevant details are omitted here due to the length limitation.
\end{example}

\begin{example}[A square domain case]\rm\label{ex:nlbc_rect}
	Set $\Omega=(-1,1)^2$ and the boundary $\partial\Omega$ is parametrized starting from the point $\x=(-1,-1)$ by a scaled arc-length in the counterclockwise direction:
	\begin{equation}\label{eq:nlbc_rect_theta}
		\x=\x(\theta)=(x_1(\theta),x_2(\theta))=
		\begin{cases}
			(4\theta/\pi-1,-1), & 0\leq\theta<\pi/2, \\
			(1,4\theta/\pi-3), & \pi/2\leq\theta<\pi, \\
			(5-4\theta/\pi,1), & \pi\leq\theta<3\pi/2, \\
			(-1,7-4\theta/\pi), & 3\pi/2\leq\theta\leq2\pi. \\
		\end{cases}
	\end{equation}
	In this case, more nontrivial solutions spring up. We simply show ten solutions obtained in Fig.~\ref{fig:nlbc_rect_u1-10} with their profiles inside the domain and on the boundary. For each solution, the information of its corresponding support space $L$, initial ascent direction $v_0$ (determined by $\rho_0$ via \eqref{eq:v0bie}) and energy functional value is listed in Table~\ref{tab:nlbc_rect}. It is observed that $u_1\sim u_5$ are positive solutions and others are sign-changing solutions.
\end{example}
\def\sfigw{.18\textwidth}
\def\sfigh{.09\textwidth}
\begin{figure}[!t]
	\centering
	\footnotesize
	\makebox[\sfigw]{$u_1$}
	\makebox[\sfigw]{$u_2$}
	\makebox[\sfigw]{$u_3$}
	\makebox[\sfigw]{$u_4$}
	\makebox[\sfigw]{$u_5$}\\
	\quad
	\includegraphics[width=\sfigw]{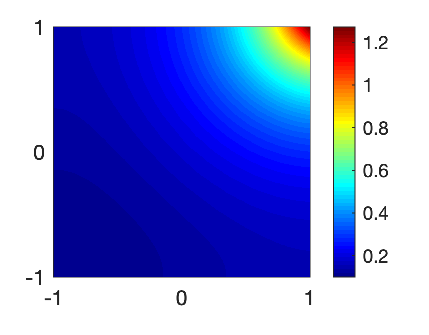}
	\includegraphics[width=\sfigw]{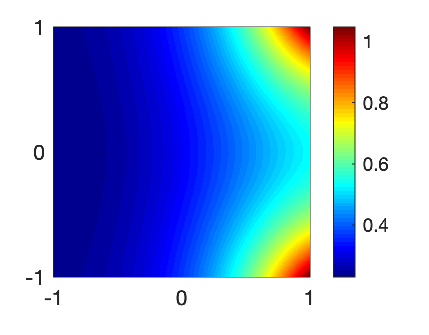}
	\includegraphics[width=\sfigw]{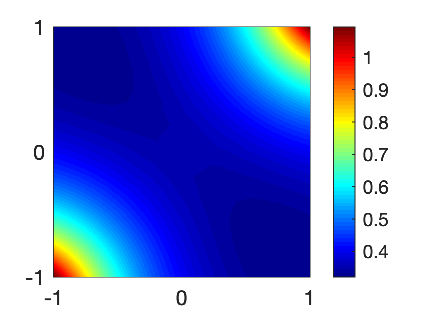}
	\includegraphics[width=\sfigw]{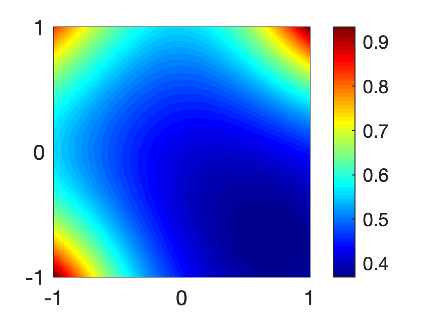}
	\includegraphics[width=\sfigw]{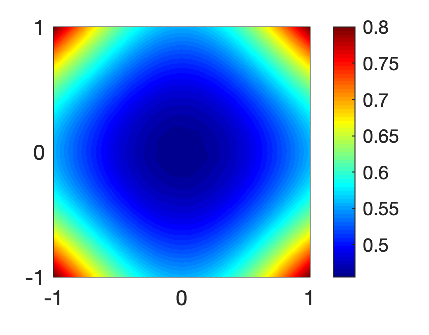} \\
	\resizebox*{\sfigw}{\sfigh}{\includegraphics{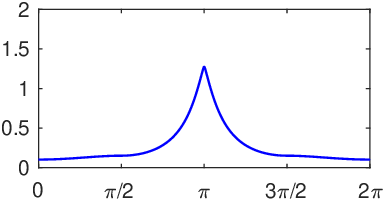}}
	\resizebox*{\sfigw}{\sfigh}{\includegraphics{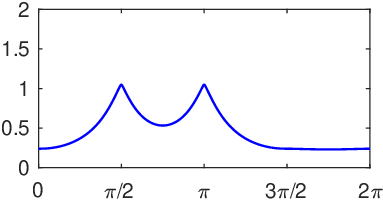}}
	\resizebox*{\sfigw}{\sfigh}{\includegraphics{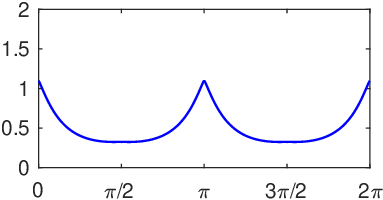}}
	\resizebox*{\sfigw}{\sfigh}{\includegraphics{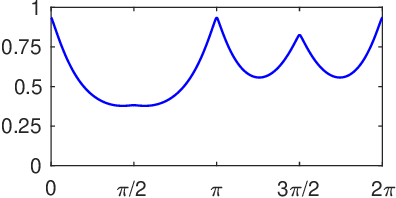}}
	\resizebox*{\sfigw}{\sfigh}{\includegraphics{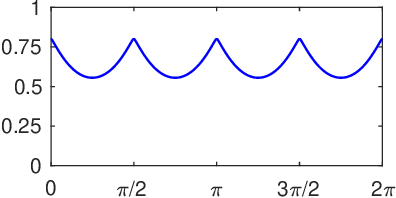}}
	\makebox[\sfigw]{$u_6$}
	\makebox[\sfigw]{$u_7$}
	\makebox[\sfigw]{$u_8$}
	\makebox[\sfigw]{$u_9$}
	\makebox[\sfigw]{$u_{10}$}\\
	\quad
	\includegraphics[width=\sfigw]{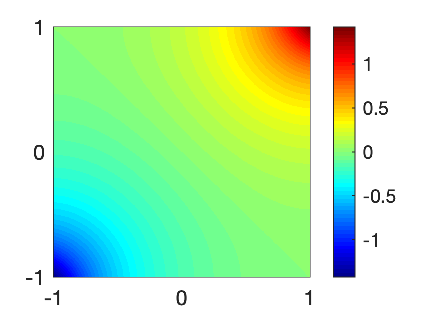}
	\includegraphics[width=\sfigw]{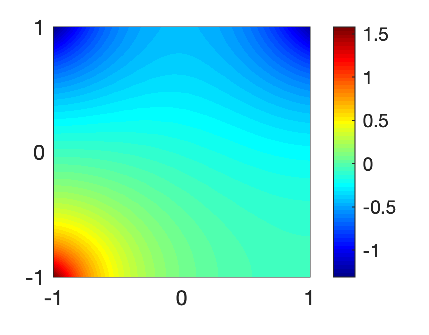}
	\includegraphics[width=\sfigw]{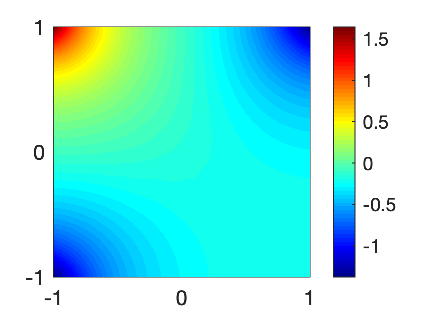}
	\includegraphics[width=\sfigw]{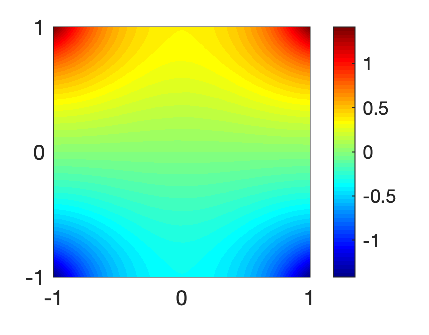}
	\includegraphics[width=\sfigw]{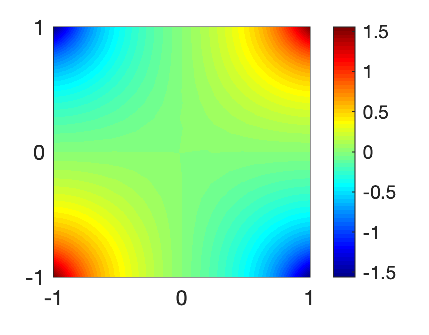} \\
	\resizebox*{\sfigw}{\sfigh}{\includegraphics{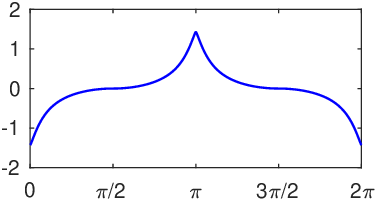}}
	\resizebox*{\sfigw}{\sfigh}{\includegraphics{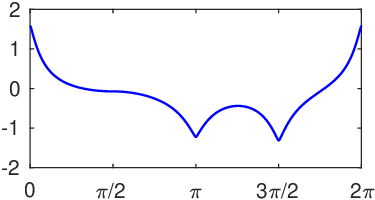}}
	\resizebox*{\sfigw}{\sfigh}{\includegraphics{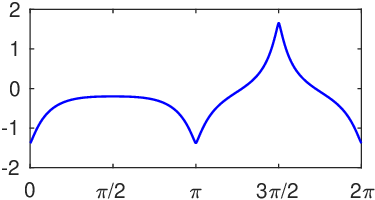}}
	\resizebox*{\sfigw}{\sfigh}{\includegraphics{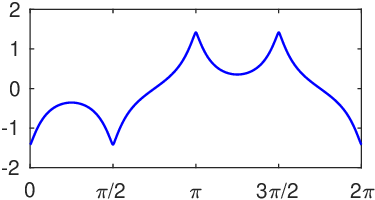}}
	\resizebox*{\sfigw}{\sfigh}{\includegraphics{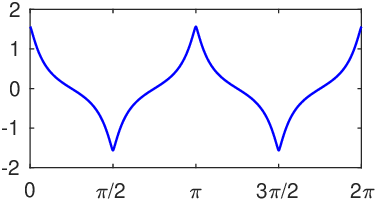}}
	\caption{Profiles of $u_1\sim u_{10}$ in Example~\ref{ex:nlbc_rect} inside the domain (first and third rows) and on the boundary (second and fourth rows). In each subplot in second and fourth rows, the horizontal axis represents the value of the  parameter $\theta$ ($0\leq\theta\leq 2\pi$) as described in \eqref{eq:nlbc_rect_theta}. Particularly, $\theta=0$ ($2\pi$), $\pi/2$, $\pi$ and $3\pi/2$ correspond to corner points $(-1,-1)$, $(1,-1)$, $(1,1)$ and $(-1,1)$, respectively.}
	\label{fig:nlbc_rect_u1-10}
\end{figure}
\begin{table}[!ht]
	\centering
	\small
	\caption{The initial information and energy functional value for each solution in Example~\ref{ex:nlbc_rect}.}
	\label{tab:nlbc_rect}
	\begin{tabular}{|c|c|c|c|c|c|}
		\hline
		$u_n$ & $u_1$ & $u_2$ & $u_3$ & $u_4$ & $u_5$ \\ \hline
		$L$ & $\{0\}$ & $\{0\}$ & $\{0\}$ & $[u_1,u_3]$ & $\{0\}$ \\ \hline
		$\rho_0(\x(\theta))$ & $1-\cos\theta$ & $1+\sin(\theta-\pi/4)$ & $1+\cos2\theta$ & $-\cos\theta$ & $1$ \\ \hline
		$E(u_n)$ & 0.2128 & 0.3068 & 0.3364 & 0.3550 & 0.3658 \\
		\hline
		\hline
		$u_n$ & $u_6$ & $u_7$ & $u_8$ & $u_9$ & $u_{10}$ \\ \hline
		$L$ & $\{0\}$ & $[u_1,u_3]$ & $[u_1,u_6]$ & $[u_1,u_2,u_3]$ & $\{0\}$ \\ \hline
		$\rho_0(\x(\theta))$ & $-\cos\theta$ & $-\cos2\theta$ & $1-\sin\theta$ & $1+\sin\theta$ & $\cos2\theta$ \\ \hline
		$E(u_n)$ & 0.5233 & 0.7474 & 0.8429 & 1.0411 & 1.2550 \\
		\hline
	\end{tabular}
\end{table}

\section{Concluding remarks}
\label{sec:con}
A novel nonmonotone LMM was proposed in this paper for finding multiple saddle points of general nonconvex functionals in Hilbert spaces by improving the traditional LMMs with nonmonotone step-size search rules. Actually, as a typical and efficient nonmonotone search rule, the normalized ZH-type nonmonotone step-size search rule was proposed for the LMM. The global convergence of the normalized ZH-type LMM was rigorously verified under the same assumptions as those in \cite{Z2017CAMC} for the normalized Armijo-type LMM. Specifically, an efficient GBBLMM was designed to speed up the convergence by combining the Barzilai--Borwein-type step-size method with the nonmonotone globalization. By applying the GBBLMM to find multiple solutions of two typical semilinear BVPs with variational structures, abundant numerical results were obtained to verify that our approach is efficient and can greatly improve the convergence rate of traditional LMMs.

As a final remark, we point out that some other feasible nonmonotone step-size search rules can be introduced to the LMM by following the lines of this paper. Typically, motivated by the GLL nonmonotone line search strategy in the optimization theory \cite{GLL1986SINUM}, one can design the following normalized GLL-type nonmonotone step-size search rule for the LMMs, that is, for $k=0,1,\ldots$, to find the step-size $\alpha=\lambda_k\rho^{m_k}$ with $m_k$ the smallest nonnegative integer satisfying
\begin{equation}\label{eq:gll}
	E(p(v_k(\alpha))) \leq \max_{1\leq j\leq\min\{M, k+1\}}E(p(v_{k+1-j}))-\sigma\alpha t_k\|g_k\|^2,
\end{equation}
where $M\geq 1$, $\sigma,\rho\in(0,1)$ and $\lambda_k>0$ are given parameters and other notations are the same as in Sect.~\ref{sec:nmlmm}. The integer $M\geq 1$ in \eqref{eq:gll} is to control the degree of the nonmonotonicity. Especially, if $M=1$, \eqref{eq:gll} degenerates to the normalized monotone Armijo-type step-size search rule stated in \eqref{eq:ak-armijo}. When $M>1$, it does not require the monotone decrease of the energy functional values. As a straightforward conclusion of Lemma~\ref{lem:armijo}, we can obtain immediately the feasibility of the GLL-type nonmonotone LMM, i.e., the LMM with the normalized GLL-type nonmonotone step-size search rule \eqref{eq:gll}. Moreover, by an analogous argument used in the proof of Theorem~\ref{thm:cvg-zhlmm}, we can prove that, under some assumptions, the sequence $\{w_k=p(v_k)\}_{k=0}^{\infty}$ generated by the GLL-type nonmonotone LMM must tend to a new critical point not in $L$ if it converges. Nevertheless, due to the more severe nonmonotonicity of the normalized GLL-type nonmonotone step-size search rule and the lack of some crucial properties for the GLL-type nonmonotone LMM, it seems that there are some potential difficulties to establish the global convergence of the whole sequence $\{w_k\}$. As the convergence result is weaker than that for the ZH-type nonmonotone LMM, we omit the details for brevity.

\paragraph{Acknowledgments.}
This work was supported by the NSFC grants 12171148 and 11771138. Liu's work was also partially supported by the NSFC grants 12101252 and 11971007. Yi's work was also partially supported by the NSFC grant 11901185, National Key R\&D Program of China (No. 2021YFA1001300) and the Fundamental Research Funds for the Central Universities 531118010207. The authors would like to thank Dr. Yongjun Yuan for his helpful discussions and useful suggestions.

\appendix

\section{Proof of Theorem~\ref{thm:lmt}}\label{app:prf-lmt}\normalsize
Set $c=\inf_{v\in \mathcal{V}_0}E(p(v))>-\infty$. Since $\mathcal{V}_0$ is a closed metric subspace of $X$ and $E(p(v))$ is continuous and bounded from below on $\mathcal{V}_0$, by the Ekeland's variational principle, for any $n\in\N_+$, there exists $v_{n}\in \mathcal{V}_0$ such that $E(p(v_{n}))<c+1/n$ and
\begin{equation}\label{eq:lmt:EEn}
	E(p(v))>E(p({v_n}))-\frac{1}{n}\|v-{v_n}\|,\quad \forall\, v\in \mathcal{V}_0\backslash\{{v_n}\}.
\end{equation}
If $g_n=\nabla E(p(v_n))\neq0$, from Lemma~\ref{lem:armijo}, for some $\sigma\in(0,1)$ and sufficiently small $\alpha>0$, we have $v_n(\alpha)=\frac{v_n-\alpha g_n}{\|v_n-\alpha g_n\|}\in\mathcal{V}_0\backslash\{v_n\}$ ($n\in\N_+$) and
\begin{equation}\label{eq:lmt:EEg}
	E(p(v_n(\alpha)))-E(p(v_n))<-\sigma\alpha t_{v_n}\|g_n\|^2 < -\sigma\delta\|g_n\|\|v_n(\alpha)-v_n\|,
\end{equation}
where \eqref{eq:va-v} and the assumption (ii) are used.
Combining \eqref{eq:lmt:EEg} and \eqref{eq:lmt:EEn} with $v=v_n(\alpha)$, one gets
\begin{equation}\label{eq:lmt:dEn}
	\|E'(p(v_n))\|_*=\|\nabla E(p(v_n))\|<\frac{1}{\sigma\delta n},\quad n\in\N_+,
\end{equation}
where $\|\cdot\|_*$ denotes the dual norm to $\|\cdot\|$. Obviously, \eqref{eq:lmt:dEn} also holds for $g_n=\nabla E(p(v_n))=0$. Hence, $E(p(v_n))\to c$ and $E'(p(v_n))\to0$ as $n\to\infty$. Then, by virtue of the (PS) condition, $\{p(v_n)\}$ possesses a subsequence $\{p(v_{n_i})\}$ converging to some $u_*\in X$ that satisfies $E'(u_*)=0$ and $E(u_*)=c$.

Next, we will show that there exists $v_*\in \mathcal{V}_0$ such that $u_*=p(v_*)$. Denote $t_n=t_{v_n}$ and $w_n^L=w_{v_n}^L$, then $p(v_n)=t_n v_n+w_n^L$. By employing orthogonal decompositions $v_n=v_n^\bot+v_n^L$ and $u_*=u_*^\bot+u_*^L$, where $v_n^\bot,u_*^\bot\in L^\bot$, $v_n^L,u_*^L\in L$, we arrive at
\[ \|p(v_{n_i})-u_*\|^2=\|t_{n_i}v_{n_i}^\bot-u_*^\bot\|^2+\|t_{n_i}v_{n_i}^L+w_{n_i}^L-u_*^L\|^2\to0,\quad \mbox{as }i\to\infty, \]
which immediately yields $t_{n_i}v_{n_i}^\bot\to u_*^\bot$ as $i\to\infty$. Since $\{v_{n_i}\}\subset\mathcal{V}_0$, there exists $\tau_{n_i}\in[0,1]$ such that $v_{n_i}^L=\tau_{n_i}v_0^L$ and
\begin{equation}\label{eq:lmt:vni}
	v_{n_i}=v_{n_i}^\bot + \tau_{n_i}v_0^L\in \mathcal{V}_0\subset S\backslash L,\quad \forall i.
\end{equation}
Then, it is clear that $0<a_0:= \sqrt{1-\|v_0^L\|^2}\leq \sqrt{1-\tau_{n_i}^2\|v_0^L\|^2}=\|v_{n_i}^\bot\|\leq1$. Thus, there is a subsequence $\{v_{n_i'}\}\subset\{v_{n_i}\}$ such that $\|v_{n_i'}^\bot\|\to a_*$ as $i\to\infty$ for some $a_*\in[a_0,1]$, and therefore
\[ t_{n_i'}=\frac{\|t_{n_i'}v_{n_i'}^\bot\|}{\|v_{n_i'}^\bot\|}\to\frac{\|u_*^\bot\|}{a_*},\quad \mbox{as }i\to\infty. \]
Further, the assumption (ii) leads to $\|u_*^\bot\|\geq\delta a_*>0$, and we have
\begin{equation}\label{eq:lmt:vnip-orth}
	v_{n_i'}^\bot=\frac{1}{t_{n_i'}}(t_{n_i'}v_{n_i'}^\bot) \to \frac{a_*}{\|u_*^\bot\|}u_*^\bot,\quad \mbox{as }i\to\infty.
\end{equation}
Note that, due to the fact $\{\tau_{n_i'}\}\subset[0,1]$, there is a subsequence $\{v_{n_i''}\}\subset\{v_{n_i'}\}$ such that
\begin{equation}\label{eq:lmt:vnipp-tau}
	\tau_{n_i''}\to\tau_*,\quad \mbox{as }i\to\infty,
\end{equation}
for a $\tau_*\in[0,1]$. Since $\{v_{n_i''}\}\subset\{v_{n_i'}\}\subset\{v_{n_i}\}$ and $\mathcal{V}_0$ is closed, combining \eqref{eq:lmt:vni}-\eqref{eq:lmt:vnipp-tau} yields 
\[ v_{n_i''}=v_{n_i''}^\bot + \tau_{n_i''}v_0^L\to v_*:=\frac{a_*}{\|u_*^\bot\|}u_*^\bot+\tau_*v_0^L\in \mathcal{V}_0,\quad \mbox{as }i\to\infty. \]
Then, the continuity of $p$ and $p(v_{n_i''})\to u_*$ as $i\to\infty$ imply $u_*=p(v_*)$.

Finally, recalling above facts that $E(u_*)=c$, $E'(u_*)=0$ and $\|u_*^\bot\|\geq\delta a_*>0$, we conclude that $p(v_*)\notin L$ is a critical point of $E$ and $E(p(v_*))=\inf_{v\in \mathcal{V}_0}E(p(v))$. The proof is finished.

\section{Two adaptive strategies for the BB-type step-sizes}\label{app:adapBB}
Motivated by the adaptive BB methods in optimization theory in Euclidean spaces developed in \cite{DHL2019COA} and \cite{HDL2021SIOPT}, we provide the following two adaptive strategies of BB-type step-sizes to compute $\lambda_k$ ($k\geq1$) in {Step~4} of Algorithm~\ref{alg:gbblmm}:
\begin{itemize}
\item {\bfseries Adaptive strategy I (Adap1).} Compute $\lambda_k$ according to \eqref{eq:lambdak-bb} with
\begin{align}
\alpha_k^{\text{BB}}&=
\begin{cases}
\alpha_k^{\text{BB2}}, & \mbox{if } \mathrm{mod}(k,k_s)=0,\\
\tilde{\alpha}_k, & \mbox{otherwise},
\end{cases}
\end{align}
where $k_s>1$ is a given positive integer and $\tilde{\alpha}_k\in[\alpha_k^{\text{BB1}},\alpha_k^{\text{BB2}}]$ is defined as
\begin{align}
\tilde{\alpha}_k=\begin{cases}
\alpha_k^{\text{BB1}}, & \mbox{if } \alpha_{k-1}\leq\alpha_k^{\text{BB1}},\\
\alpha_k^{\text{BB2}}, & \mbox{if } \alpha_{k-1}\geq\alpha_k^{\text{BB2}},\\
\alpha_{k-1}, & \mbox{otherwise},
\end{cases}
\end{align}
with $\alpha_{k-1}$ the step-size used in the previous iterative step and $\alpha_k^{\text{BB1}}$ and $\alpha_k^{\text{BB2}}$, respectively, the short and long BB-type step-sizes given in \eqref{eq:bbss}. In our experiments the parameter $k_s$ is fixed as $k_s=8$, as suggested by the numerical results in \cite{DHL2019COA}.

\item {\bfseries Adaptive strategy II (Adap2).} Compute $\lambda_k$ according to \eqref{eq:lambdak-bb} with $\alpha_1^{\text{BB}}=\alpha_1^{\text{BB2}}$ and
\begin{align}
\alpha_k^{\text{BB}}=
\begin{cases}
\min\big\{\widehat{\alpha}_k, \alpha_k^{\text{BB1}},\alpha_{k-1}^{\text{BB1}}\big\}, & \mbox{if } \alpha_k^{\text{BB1}}/\alpha_k^{\text{BB2}}<\tau_k, \\
\alpha_k^{\text{BB2}}, & \mbox{otherwise},
\end{cases}\quad (k\geq2),
\end{align}
where $\widehat{\alpha}_k$ ($k\geq2$) is defined as \cite{HDL2021SIOPT}
\[ 
\widehat{\alpha}_k=\frac{2}{b_k+\sqrt{b_k^2-4a_k}},\;
a_k=\frac{\alpha_{k-1}^{\text{BB1}}-\alpha_k^{\text{BB1}}}{\alpha_{k-1}^{\text{BB1}}\alpha_k^{\text{BB1}}(\alpha_{k-1}^{\text{BB2}}-\alpha_k^{\text{BB2}})}, \;
b_k=\frac{\alpha_{k-1}^{\text{BB2}}\alpha_{k-1}^{\text{BB1}}-\alpha_k^{\text{BB2}}\alpha_k^{\text{BB1}}}{\alpha_{k-1}^{\text{BB1}}\alpha_k^{\text{BB1}}(\alpha_{k-1}^{\text{BB2}}-\alpha_k^{\text{BB2}})},
\]
and $\tau_k$ dynamically updated by
\[ \tau_{k+1}=\begin{cases}
\tau_k/\gamma, & \mbox{if } \alpha_k^{\text{BB1}}/\alpha_k^{\text{BB2}}<\tau_k, \\
\tau_k\gamma, & \mbox{otherwise}
\end{cases} \]
for some given $\tau_0\in(0,1)$ and $\gamma>1$. Following the numerical computations in \cite{HDL2021SIOPT}, $\tau_0=0.2$ and $\gamma=1.02$ are adopted in our numerical experiments.
\end{itemize}


\end{document}